\documentclass[english]{article}
\usepackage[T1]{fontenc}
\usepackage[latin9]{inputenc}
\usepackage{mathrsfs}
\usepackage{amsmath}
\usepackage{amsthm}
\usepackage{amssymb}
\usepackage{graphicx}
\usepackage{esint}
\usepackage[all]{xy}
\usepackage{epsfig}
\usepackage{graphics}

\makeatletter

\providecommand{\tabularnewline}{\\}
\numberwithin{equation}{section}
  \theoremstyle{definition}
  \newtheorem{defn}{\protect\definitionname}[section]
  \theoremstyle{plain}
  \newtheorem{thm}{\protect\theoremname}[section]
  \theoremstyle{plain}
  \newtheorem{prop}{\protect\propositionname}[section]
  \theoremstyle{definition}
  \newtheorem{example}{\protect\examplename}[section]

\makeatother

\usepackage{babel}
  \providecommand{\definitionname}{Definition}
  \providecommand{\examplename}{Example}
  \providecommand{\propositionname}{Proposition}
\providecommand{\theoremname}{Theorem}

\begin{document}
\pagenumbering{roman}
\title{An Extended Goodwin Model with Horizontal Trade: A Sheaf Theoretic
Approach}

\author{Philip Coyle}
\maketitle
\begin{abstract}
The Goodwin model of endogenous growth looks to study the dynamic
interaction between employment rate and worker's share of national
income in an economy. The model is simplistically and elegantly described
by a set of differential equations that predicts cyclic behavior
between the two variables in an economy. While this model is simplistic,
and most likely does not accurately represent reality, the mathematical
modeling of cyclic behavior is attractive to economists. Cycles are
at the heart of many macroeconomic theories and a mathematical model
allows for future predictions to be made. Thus, over the years, it
has been updated and extended. Ishiyama (2001) takes one such approach
at updating the Goodwin model. He considers two countries engaged
in horizontal trade. Through the lens of sheaf theory, this paper
describes Ishiyama's complex model through various dependency diagrams.
Sheaves allow us to encode all information reflected in the equations
of the model into a dependency diagram, yielding a visual representation
of the variable relationship structure. These dependency diagrams
are powerful, and much of the analysis typically done on equations
can be done on the diagrams themselves. Further, these dependency
diagrams allow for a new way to consider complex models, such as Ishiyama's
model. More specifically, it also allows us to analyze his system
in a way not previously done. New questions regarding local sections
of this sheaf and their possible extensions to global sections are
considered. These questions lend themselves to unique analysis about
the system. They also provide a practical way to check the accuracy
of Ishiyama's model, which has many obvious benefits. It is meaningful
to conduct this approach to a system of equations given its novelty,
its applicability, and its importance for modeling. 
\end{abstract}
\pagebreak{}

\tableofcontents{}

\listoffigures

\listoftables

\pagebreak{}

\section{Introduction}
\pagenumbering{arabic}
To model complex phenomena, mathematicians, physicists, engineers,
chemists, and economists have long turned toward differential equations
to provide insight and analysis. Because differential equations describe
phenomena in a functional way, they can be used to make predictions
about how systems evolve over time. Before the rise of computers
and the ease to collect and analyze data, models built using differential
equations were the first type of predictive machine; and were greatly
successful, too. Such a simple concept was used to great effectiveness
by Newton to model the trajectories of heavenly bodies. Many of his
equations are still widely used today.

Economics, too, have come to understand the power of differential
equations to model and predict phenomena. Economists, however, are
not the only people who seek to understand and predict how a country's
economy grows over time. Such insight would be useful to politicians
and businessmen alike. Provided the models are accurate, they can
be quite informative. Therefore, there has been a great deal of interest
in the rigor and accuracy with which these model economic phenomena
which we observe.

The behavior of economic systems have been viewed in a few distinct
ways over time. The first models claimed that markets obtain stable
equilibrium over time. Random shocks, such as an oil crisis, are observable,
but eventually the systems return to equilibrium. Later models assumed
more complex growth behavior. These models assumed that growth paths
are not steady, but rather cyclic, and equilibria are effected by
past motions. This appears to be more in line with reality, too.

In 1967, Goodwin proposed a simple model which describes the dynamic
relationship between the employment rate and worker's share of national
income. While Goodwin admitted that it was a ``quite unrealistic
model of cycles in growth rates,'' the beauty of his model lies in
its simplicity, while still yielding interesting cyclic dynamics.
Over time, economists have attempted to update Goodwin's model with
two goals in mind. The first is to make it more applicable to reality and
herein lies the second point of interest: how do the dynamics of the
updated system change? This paper will ultimately focus on one of
these updated models. Ishiyama \cite{Dummy1} proposed a model which extends
the original Goodwin beyond its  closed country assumption
to an open one by building horizontal trade into the model. It is,
at its core, a multi-model system. His new model demonstrates the
emergence of chaotic behavior by allowing for horizontal trade. But
how accurate is Ishiyama's model? What else can be said about Ishiyama's
extended model? A branch of mathematics called sheaf theory may provide
some answers. 

The recent application of sheaf theory to multi-model systems looks
to ease the burden of constructing and analyzing complex models.
Sheaves provide a ``tool box'' for constructing predictive model's
described by a system of equations. Sheaves also provide a way to
check the models accuracy. The power of sheaves comes about since
smaller and easier to construct models can be systematically stitched
together to form a larger, more complex one. Sheaves are represented
diagrammatically as dependency diagrams, where arrows relating variables
to each other represent actual functions. Any analysis done to the
sheaf is equivalent to analysis done on the system of equations. Ishiyama
extends the Goodwin Model to horizontal trade, whereby he takes two
separate countries and effectively ``stitches'' them together via
a trading scheme. At its core, Ishiyama's model is that of a multi-model
system and sheaves have the power to shed light on his model. This
paper will use Ishiyama's proposed model for international horizontal
trade and construct a number of dependency diagrams and sheaves. 
Each will be analyzed using the applicable tools. This analytical approach 
to a system of equations is meaningful given its novelty, its applicability, 
and its importance for modeling. 

This paper is organized in the following way: Section 2 discusses
nonlinear dynamical systems and introduces important concepts relating
to the Goodwin model. Section 3 introduces the predator-prey model.
Section 4 discusses the Goodwin model and various updates to it. Section
5 provides background information on sheaf theory and how to construct
and analyze sheaves. These four sections provide a comprehensive and
sufficient amount of background information to make the final section
understandable, in which the extended Goodwin model is introduced,
sheafified, and analyzed. The final section concludes and discusses
areas of further research.

\section{Nonlinear Dynamical Systems}

\subsection{Preliminary Concepts}

We begin our background discussion with nonlinear dynamical systems.
It is not a stretch to assume that the world is not linear in nature;
thus to model phenomena that take place in it, it is natural to craft
nonlinear models. While this process allows us to model phenomena
in a better way, there are some tradeoffs; most notably, most nonlinear
dynamical systems cannot be solved in closed form. What is available,
however, is an analysis of the qualitative behavior of the system.
It is possible to determine under certain conditions whether a dynamical
system exhibits a closed orbit or displays other dynamic behavior.
The rest of this section is devoted to providing some background on
the concept of closed orbits in dynamical systems. We will introduce
key definitions, concepts, and theorems.

Since we are dealing with dynamical systems of differential equations
rather than difference equations, we will be exclusively dealing with
continuous time. We begin by considering the \emph{n}-dimensional
ordinary differential equations system defining the motion of the
state variables $x_{i},i=1,\dots,n$ 
\[
\dot{x}_{1}=f_{1}(x_{1},\dots,x_{n}),
\]
\begin{equation}
\vdots\label{eq:2.1.1}
\end{equation}
\[
\dot{x}_{n}=f_{n}(x_{1},\dots,x_{n}).
\]
Alternatively, this can be written more succinctly in vector notation
as 
\begin{equation}
\dot{\boldsymbol{x}}=\boldsymbol{f(x)}\mbox{\mbox{ where \ensuremath{\boldsymbol{x}\in\mathfrak{\mathfrak{\mathcal{W}}}\subseteq\mathbb{R}^{n}}}}.\label{eq:2.1.2}
\end{equation}
Here, $\mathcal{W}$ is an open subset of $\mathbb{R}^{n}$ and the
dot over a variable denotes the derivative with respect to time. 
\begin{defn}
Consider a subset $\mathcal{S}$ of the domain $D$, under the mapping
$T:D\rightarrow D$. An \emph{invariant set} is the subset $\mathcal{S}$
such that for all $x\in\mathcal{S}$, $T(x)\in\mathcal{S}$.
\end{defn}
A \emph{solution curve}, \emph{trajectory},or \emph{orbit}
is defined as $\phi_{t}(\boldsymbol{x}(0))$. For a certain set of initial conditions $\boldsymbol{x}(0)$
given, $\phi_{t}(\boldsymbol{x}(0))$ describes the values of $\boldsymbol{x}$
at \emph{$t$}. The \emph{flow }of the system, $\phi_{t}(\boldsymbol{x}):\mathbb{R}^{n}\rightarrow\mathbb{R}^{n}$,
describes the future development of all $\boldsymbol{x}(0)\in\mathcal{W}$. 

An important concept in nonlinear dynamical systems is that of an
attractor. An attractor is a set of numerical values that a system
tends to evolve toward. This happens for a wide variety of starting
conditions of the system. More formally:
\begin{defn}
A closed invariant set $\mathcal{A\subset\mathcal{W}}$ is called
an \emph{attracting set }if there is some neighborhood $\mathcal{U}$
of $\mathcal{A}$ such that $\phi_{t}(\boldsymbol{x}(t))\in\mathcal{U}$
for all $t>0$ and $\underset{t\rightarrow\infty}{lim}\phi_{t}(\boldsymbol{x}(t))=\mathcal{A}$
for all $x\in\mathcal{U}$.
\end{defn}
Additionally, a \emph{repelling set} is simply defined by letting
$t\rightarrow-\infty$ in Definition 2.1. Thus, an attracting set
is simply a set to which trajectories starting at an initial value
in a neighborhood of the set will eventually converge. The \emph{basin
of attraction $\mathcal{A}$ }is therefore the set of all initial
points which are attracted to $\mathcal{A}$. 
\begin{defn}
Let $\mathcal{U}$ be a neighborhood of an attracting set $\mathcal{A}$.
The \emph{basin of attraction} $\mathcal{B}(\mathcal{A})$ is the
region of the phase space $\mathcal{U}$ such that any initial condition
in that region will eventually be iterated into the attractor.
\end{defn}
Attracting sets can also be detected by a \emph{trapping region }\cite{Wiggins90}
\begin{defn}
A closed and connected set $\mathcal{D}$ is a \emph{trapping region}
if $\phi_{t}(\mathcal{D})\subset\mathcal{D}$ for all $t\ge0.$
\end{defn}
There are two types of attractors which we will focus on in this section:
fixed-point attractors and closed orbits.\emph{ }

\subsubsection{Fixed Point Attractors}

A \emph{fixed point} of a function is a point that is mapped to itself
by the function. In other words, if we observe the dynamical system
evolving over time, the final state at which the system is in corresponds
to an attracting fixed point of the function describing the system.
Connecting fixed points to attractors, we can define a fixed point
in the following way:
\begin{defn}
An invariant set $\mathcal{A}$ that consists of only a single element
is called a \emph{fixed point}. 
\end{defn}
A fixed point, or equilibrium point, is a type of attractor that has
been featured heavily in most economic research. When discussing nonlinear
systems, local and global stability properties are considered. \\
\emph{}\\
\emph{Local Stability of Fixed Points}\\
\\
Let $\boldsymbol{x^{*}}=(x_{1}^{*},\dots,x_{n}^{*})$ be a fixed point
of (2.1) \emph{i.e.} $\dot{\boldsymbol{x}}=0=\boldsymbol{f}(\boldsymbol{x}^{*})$.
We introduce the following two local stability concepts, which are
relevant in economic fixed point analysis: 
\begin{defn}
A fixed point is \emph{locally stable} (\emph{locally Lyapunov stable})
if for every $\epsilon>0$ there exits a $\delta>0$ such that for
all $t,$ $|\phi_{t}(\boldsymbol{x}(0))-\boldsymbol{x^{*}}|\le\epsilon$
whenever $|\boldsymbol{x}(0)-\boldsymbol{x}|\le\delta$. 
\end{defn}
We also consider asymptotically stable fixed points. These points
are asymptotically stable if they are elements of the attracting set
$\mathcal{A}$.
\begin{defn}
A fixed point is \emph{asymptotically stable }if for every $\epsilon>0$
there exits a $\delta>0$ such that $\underset{t\rightarrow\infty}{lim}|\phi_{t}(\boldsymbol{x}(0))-\boldsymbol{x^{*}}|=0$
whenever $|\boldsymbol{x}(0)-\boldsymbol{x}|\le\delta$.
\end{defn}
\begin{figure}

\begin{centering}
\includegraphics[scale=0.3]{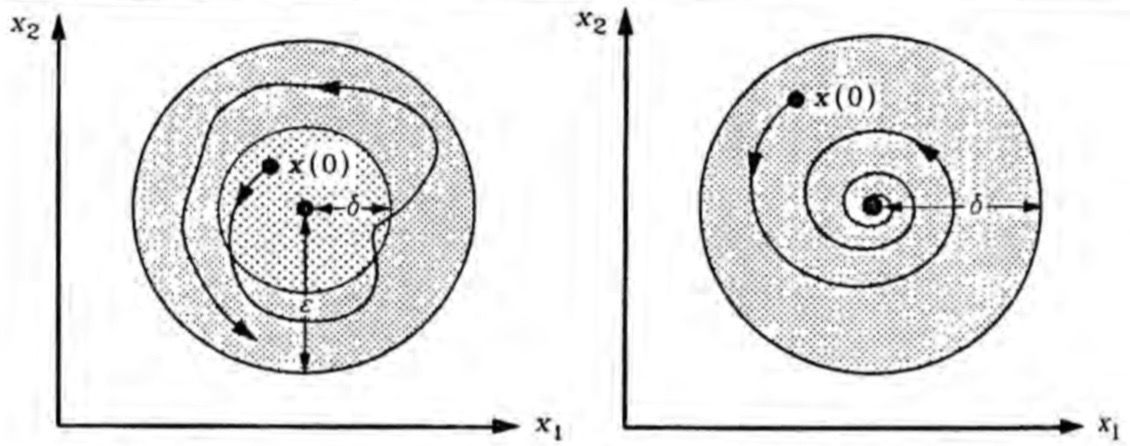}\caption{Lyapunov (Left) and Asymptotic (Right) Stability \cite{Lorenz93}}

\par\end{centering}

\end{figure}

Figure 1 illustrates these two definitions. Note that in order to
be locally stable, a trajectory starting in the neighborhood of $\boldsymbol{x^{*}}$
is \emph{required} to stay within an $\epsilon-$neighborhood. Alternatively,
a trajectory is asymptotically stable if it starts in a $\delta$-neighborhood
of $\boldsymbol{x^{*}}$ and converges toward the fixed point. As
time progresses forward, the distance, measured in Euclidian space,
between the points $\phi_{t}(\boldsymbol{x}(0))$ and $\boldsymbol{x^{*}}$
decreases. \\
\emph{}\\
\emph{Global Stability of Fixed Points}\\
\\
It is important to make the distinction between local and global stability
of fixed points of dynamical systems. Unlike in linear systems, where
local stability points imply global stability points, nonlinear systems
can be characterized by multiple fixed points which are locally asymptotically
stable or unstable. Below, we define global asymptotic stability in
a similar way to local asymptotic stability, however with one modification:
\begin{defn}
A fixed point is \emph{globally asymptotically stable }if it is stable
and $\underset{t\rightarrow\infty}{lim}|\phi_{t}(\boldsymbol{x}(0))-\boldsymbol{x^{*}}|=0$
for every $\boldsymbol{x}(0)$ in the domain of (2.1).
\end{defn}
A useful tool to determine the global stability of a fixed point is
the concept of a \emph{Lyapunov function}: 
\begin{thm}
\textbf{\cite{Lyapunov49}} Let $\boldsymbol{x^{*}}$ be a fixed point
of a differential equation system and let $V:\mathcal{U\rightarrow\mathbb{R}}$
be a differentiable function on some neighborhood $\mathcal{U\subset\mathcal{W\subset\mathbb{\mathbb{R}}}}^{n}$
of $\boldsymbol{x^{*}}$ such that the following hold:\emph{}\\
\emph{(i)} $V(\boldsymbol{x^{*}})=0$ and $V(\boldsymbol{x})>0$ if
$\boldsymbol{x}\neq\boldsymbol{x^{*}}$\emph{}\\
\emph{(ii)} $\dot{V}(\boldsymbol{x})\leq0$ in $\mathcal{U}-\{\boldsymbol{x^{*}}\}$
\\
Then $\boldsymbol{x^{*}}$ is stable. Further if \emph{}\\
\emph{(iii)} $\dot{V}(\boldsymbol{x})<0$ in $\mathcal{U}-\{\boldsymbol{x^{*}}\}$
\\
then $\boldsymbol{x^{*}}$ is asymptotically stable.
\end{thm}
It is important to observe that the neighborhood $U\subset\mathcal{W}$
can be arbitrarily large. As a result, a fixed point is globally asymptotically
stable if conditions (\emph{i-iii)} are satisfied on the domain of
(2.2).

\subsubsection{Cyclic Attractors}

The second kind of attractor is a cyclic one. The following discussion
on cyclic attractors concentrates on attractors in the form of closed
orbits. We say a point $\boldsymbol{x}$ is in a \emph{closed orbit}
if there exists a $t\neq0$ such that $\phi_{t}(\boldsymbol{x})=\boldsymbol{x}$.
A \emph{limit cycle} occurs when a closed orbit is an attractor. 
\begin{defn}
\cite{Lorenz93} A closed orbit $\Gamma$ is called a \emph{limit
cycle} if there is a tubular neighborhood $\mathcal{U}(\Gamma)$ such
that for all $\boldsymbol{x}\in\mathcal{U}(\Gamma)$, any flow $\phi_{t}(\boldsymbol{x})$
approaches the closed orbit.
\end{defn}
In most nonlinear economic applications, we wish to comment on the
global behavior of the dynamical system. However, restrictions arise
since it is only possible to completely categorize the global behavior
of a dynamical system in two dimensions. Classifying behavior in higher
dimensions is not as easy. To be able to classify the global behavior
of a dynamical system in two dimensions, we turn to the Poincaré-Bendixson
Theorem. We begin with a definition:
\begin{defn}
A \emph{$\omega$-limit set }of a point $\boldsymbol{x}\in\mathcal{W}$
is the set of all points $l\in\mathcal{W}$ with the property that
there exists a sequence $t_{i}\rightarrow\infty$ such that $\lim_{i\rightarrow\infty}\phi_{t_{i}}(\boldsymbol{x})=l$.
The \emph{$\alpha$-limit set} is defined in the same way, however
the sequence $t_{i}\rightarrow-\infty$. 
\end{defn}
As an example note that in $\mathbb{R}^{2}$, there are three different
types of limit sets that exist:
\begin{enumerate}
\item Fixed Point Attractors
\item Limit Cycles
\item Saddle loops 
\end{enumerate}
The first two types of limit sets have already been discussed. An
example of a saddle loop is described in Figure 2.  The dynamical
system has two unstable fixed points (point A and point C) and a saddle
as a third fixed point (point B). The $\omega$-limit set in this
example is the union of the two loops, or the trajectory that leaves
the saddle and returns to it, \emph{and }the saddle point itself.
Note that saddle loops can enclose closed orbits. 

\begin{figure}
\begin{centering}
\includegraphics[scale=0.35]{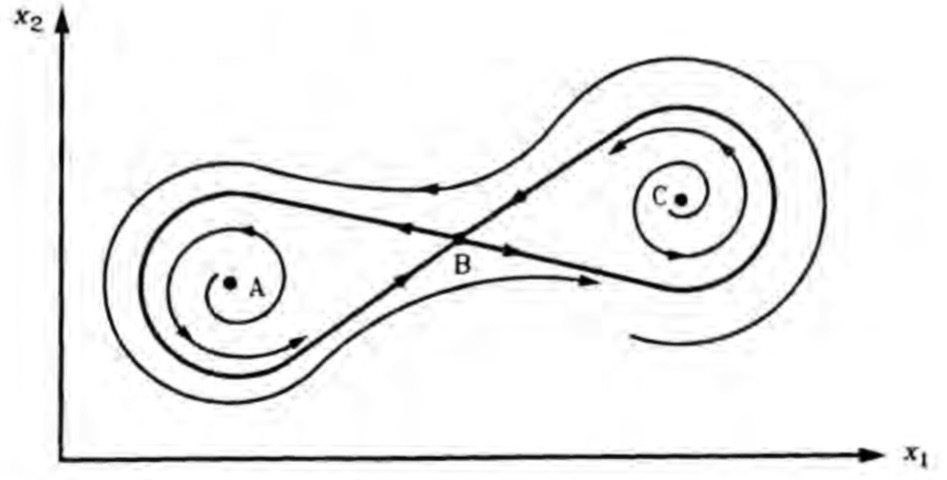}\caption{A Saddle Loop \cite{Lorenz93}}

\par\end{centering}

\end{figure}

The Poincaré-Bendixson Theorem provides sufficient conditions for
the existence of limit cycles in sub-areas of the plane. To begin,
consider the two-dimensional differential equation system:
\begin{equation}
\begin{array}{c}
\dot{x}_{1}=f_{1}(x_{1},x_{2}),\\
\dot{x}_{2}=f_{2}(x_{1},x_{2}),
\end{array}\label{eq:3}
\end{equation}
and let the initial point $\boldsymbol{x}(0)=(x_{1}(0),x_{2}(0))$
be located in an invariant set $\mathcal{D}\subset\mathbb{R}^{2}$.
Fixed point attractors, limit cycles,
and saddle loops are all possible, when the set contains limit sets. The Poincaré-Bendixson Theorem
discriminates between these different types:
\begin{thm}
\textbf{(Poincaré-Bendixson) }A non-empty compact limit set of a $C^{1}$
dynamical system in $\mathbb{R}^{2}$ that contains no fixed points
is a closed orbit.
\end{thm}
While the fixed point has been excluded from the limit set in $\mathcal{D}$,
a closed orbit in $\mathbb{R}^{2}$ will always enclose a fixed point 
\begin{thm}
A closed trajectory of a continuously differentiable dynamical system
in $\mathbb{R}^{2}$ must enclose a fixed point with $\dot{x}_{1}=\dot{x}_{2}=0$. 
\end{thm}
To summarize, the steps outlined are a procedure in applying the Poincaré-Bendixson
Theorem to a dynamical system in $\mathbb{R}^{2}$:
\begin{itemize}
\item Locate a fixed point of the dynamical system and examine its stability
properties.
\item If the fixed point is unstable, proceed to search for an invariant
set $\mathcal{D}$ enclosing the fixed point. When a closed orbit does
not coincide with the boundary of $\mathcal{D}$, the vector field
described by the function $f_{1}$ and $f_{2}$ must point into the
interior of $\mathcal{D}$.
\end{itemize}
The search for the set $\mathcal{D}$ is the difficult aspect of applying
the Poincaré-Bendixson Theorem to dynamical systems. However, it is
easy to exclude the existence of closed orbits in a system like (2.3);
consider $\mathcal{S}$: a simply connected domain in $\mathcal{W}\subseteq\mathbb{R}^{2}$.

The Poincaré-Bendixson Theorem provides sufficient conditions for
the existence of closed orbits in a set $\mathcal{D}$. However the
number of these orbits is not known. It is possible that more than
a single orbit exists, and if several cycles exist it is impossible
that all cycles are limit cycles. Given that a fixed point is unstable,
the innermost cycle in $\mathcal{D}$ is stable. Yet the most serious
disadvantage of the Poincaré-Bendixson Theorem is that it is restricted
to dynamical systems of dimension two only. Analogous theorems for
higher dimensions don't exist, severely limiting the analysis that
can be done.

\subsection{Conservative Dynamical Systems}

If a system has a single limit cycle, then the trajectories starting
at initial points in the basin of attraction are attracted by this
cycle. In addition to these limit cycle systems, another type of dynamical
system exists - conservative dynamical systems. The fundamental property
of a conservative system is the existence of a function for the dependent
variables which is constant in motion. This plays the equivalent role
of ``energy'' in physical systems. This dynamical system is able
to generate oscillations, however it is characterized by different
dynamic behavior. 

Consider a two-dimensional dynamical system:
\begin{equation}
\begin{array}{c}
\dot{x}=f_{1}(x,y)\\
\dot{y}=f_{2}(x,y)
\end{array}\label{eq: 4}
\end{equation}
 The Jacobian Matrix is as follows:
\[
\boldsymbol{J}=\left(\begin{array}{cc}
\frac{\partial f_{1}}{\partial x} & \frac{\partial f_{1}}{\partial y}\\
\frac{\partial f_{2}}{\partial x} & \frac{\partial f_{2}}{\partial y}
\end{array}\right).
\]
Let the determinant for $\boldsymbol{J}$ be positive for all $(x,y)$.
Note that the sign of the trace of the Jacobian - \emph{i.e. }the
sum of the elements in the main diagonal - plays an important role
in determining the kind of oscillating behavior of a two dimensional
dynamical system. We would like to be able to assign a qualitative
description to the meaning of the trace of $\boldsymbol{J}$. There
are three possibilities: the trace is positive, negative, or zero. 

If the trace is positive, the fixed point is considered unstable,
\emph{i.e.} there exists a tendency away from the fixed point in all
directions. We can think of this as a sort of negative friction - every
point in the phase space would spiral outward, and no closed orbit
would exist. Note that by Theorem 2.3, the trace of the Jacobian needs
to change sign if limit cycles are to be generated. Thus a negative
trace represents this idea of positive friction. Formerly exploding
behavior will be dampened for all points that are sufficiently far
away from the equilibrium. Therefore, a closed orbit emerges where
the exploding and imploding forces \emph{both} tend toward zero. Dynamical
systems that exhibit this type of behavior are called \emph{dissipative
}systems. \\

\begin{defn}
\cite{Lorenz93} A system of ordinary differential equations 
$\dot{\boldsymbol{x}}=\boldsymbol{f}(\boldsymbol{x})$ is called \emph{dissipative}
if there are numbers $R>0$ and $t_{1}>0$ such that for all solutions
$\boldsymbol{x^{*}}$of the system if $|\boldsymbol{x}(\boldsymbol{0})|\leq R$
then $|\boldsymbol{x}(\boldsymbol{t})|\leq R$ whenever
$t>t_{1}$. 
\end{defn}
The term stems from the consideration of physical systems where there
exists a permanent input of energy. That energy dissipates throughout
the system. If the energy input is interrupted, the system collapses
to its equilibrium state. 

There is another class of systems that we will consider: a \emph{conservative}
dynamical system. A conservative system is one where no friction exists,
\emph{i.e. }there is no additional input or loss of energy. Keeping
with the theme of classifying the value of the trace, the absence
of friction seen in a conservative dynamical system corresponds to
a zero trace for all points in the phase space. \\

\begin{defn}
Consider the Jacobian matrix for the dynamical system described in
(2.4). Then, a \emph{conservative} dynamical system is one where $tr(\boldsymbol{J})=0$.
\end{defn}
In geometric terms, conservative systems are characterized by the
fact that throughout the evolutionary process, an element in the phase
space changes only its shape, but the volume remains the same. Points
rotate around an elliptic fixed point and volume is conserved. In
dissipative systems, trajectories are attracted to a fixed point and
volume shrinks. Figures 3 and 4 helps further illustrate this point.
Assume that a dynamical system has infinitely many closed orbits and
that every initial point is located in such a closed orbit. If we
consider the area A in the top diagram, initial points contained in
this subset of the plane eventually move to the area B under the action
of the flow. Note that the area of A is identical to the area of B.
We call this dynamical system, thus, \emph{area preserving}. A dynamical
system that is area preserving is also conservative. This is in contrast
to a dissipative system where areas get smaller, \emph{i.e.} the trajectories
converge to an attractor. The bottom diagram shows two trajectories
which start at different initial points and approach an attractor.
The area between the two trajectories is continually getting smaller
and in the limit approaches zero as the trajectories approach the
fixed point. This is formalized in the following definition:
\begin{defn}
Consider a differential equation in $\mathbb{R}^{2}$ that implies
for some function $f(x)$ that $f'(x)=0$, then
\[
f(x)=a,
\]
where $a$ is a constant along the trajectories of the solutions.
The equation $f(x)=a$ is called the \emph{conservation law}.

\begin{figure}
\begin{centering}
\includegraphics[scale=0.3]{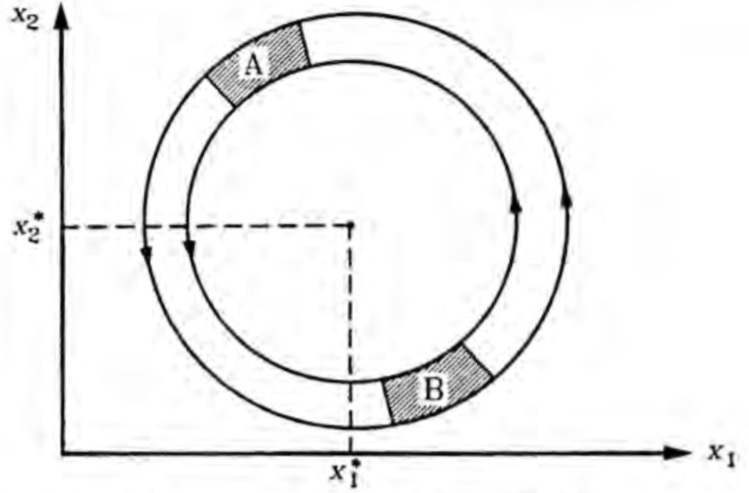}\caption{Area Preservation in a Conservative Dynamical System \cite{Lorenz93}}

\par\end{centering}

\end{figure}

\begin{figure}

\begin{centering}
\includegraphics[scale=0.3]{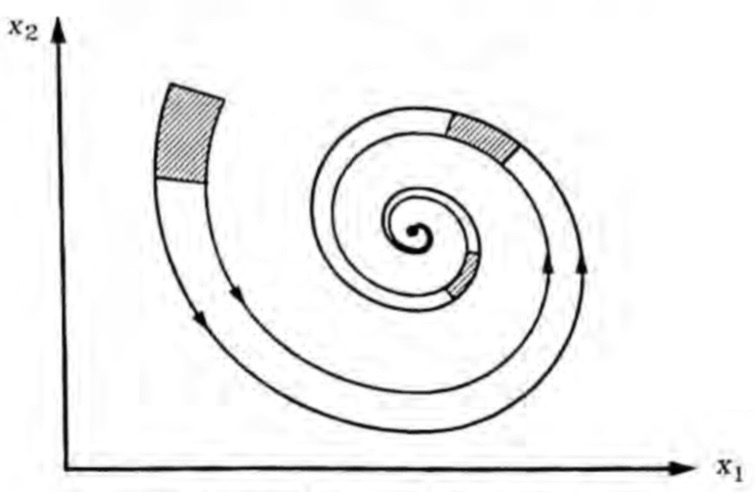}\caption{Area Contraction in a Dissipative System \cite{Lorenz93}}

\par\end{centering}

\end{figure}

\end{defn}
One such example of a conservative dynamical system in economics is
the Goodwin Class struggle model, which is another form of the Lokta-Volterra
Predator-Prey Model.

\section{The Predator-Prey Model}

We begin with a brief introduction to the predator-prey model. For
further reading, see Lotka \cite{Lotka25} and Volterra \cite{Volterra26}.
The predator-prey model is one of the first mathematical models for
biological systems. Originally, it was based on interaction between
predatory and non-predatory fish in the Adriatic sea. First, we derive
it intuitively, and then through a more rigorous mathematical analysis,
consider the system. 

In the 1920's, the observed amount of predatory fish in the Adriatic
was higher than expected. Given the destruction of many of the fisheries
during the first world war, one would anticipate that less prey in
a given environment would also lead to less predators. However, the
observations contradicted this notion. 

Volterra, a mathematical biologist, wanted to model this dynamic based
on his observations. To begin, Volterra assumed that under the absence
of a predator, \emph{i.e. }$y=0$, the growth rate of prey is given
by a constant, $a$. Further, this growth rate would be dependent
on the density of the predator population, $y$, with a linear factor
$b$. This leads to the construction of our first equation:
\[
\frac{\dot{x}}{x}=a-by
\]
where $a,b>0$. Simplifying this, we get 
\[
\dot{x}=(a-by)x.
\]
When considering the growth of the predator, Volterra assumed if there
is no prey available, \emph{i.e.} $x=0$, then the predator population
dies, which is given by a constant decay rate $-c$. Just as the prey
population depends on the density of the predator, the predator depends
on the population of the prey, leading to the following equation:
\[
\frac{\dot{y}}{y}=-c+dx
\]
where $c,d>0.$ This simplifies to
\[
\dot{y}=(-c+dx)y.
\]
 Thus we are left with a system of equations, which forms the Lotka-Volterra
predator-prey model:
\begin{equation}
\begin{array}{c}
\dot{x}=ax-bxy\\
\dot{y}=-cy+dxy
\end{array}\label{eq:5}
\end{equation}
$x$ represent the total prey population and $y$ represents the total
predator population. These equations describe the dynamical change
in both populations at a given point in time. The dependency between
the two variables is illustrated in Figure 5. This is an example of
a variable dependency graph, which will be formally defined in Definition
5.8. Notice that there are two main reinforcing loops: the stock of
prey is dependent on the net growth of prey and the stock of predators
changes with prey's net growth. This balancing loop creates oscillations.

\begin{figure}

\begin{centering}
\includegraphics[scale=0.7]{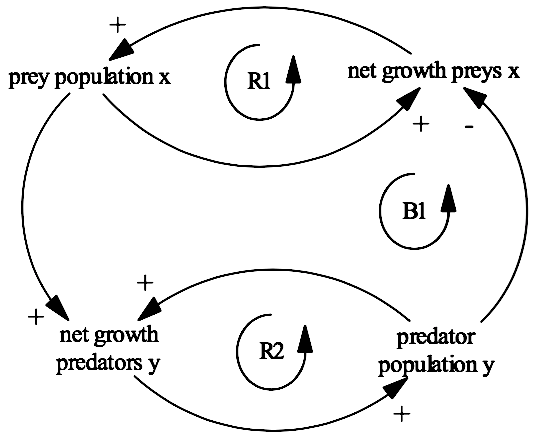}\caption{Causal Loop Diagram of Predator-Prey Model \cite{Weber05}}

\par\end{centering}

\end{figure}

A more formal mathematical consideration of the system and its dynamically
relevant attributes is now in order. Consider again the system outlined
in (3.1). It is easy to see that system (3.1) has two fixed points
when $\dot{x}=\dot{y}=0.$ These are the trivial fixed point $(x^{*},y^{*})=(0,0)$ - a saddle point - and $(x^{*},y^{*})=(\frac{c}{d},\frac{a}{b})$, a
stable center. The Jacobian matrix is:

\[
\begin{array}{ccc}
\boldsymbol{J} & = & \left(\begin{array}{cc}
\frac{\partial f_{1}}{\partial x} & \frac{\partial f_{1}}{\partial y}\\
\frac{\partial f_{2}}{\partial x} & \frac{\partial f_{2}}{\partial y}
\end{array}\right)\\
\boldsymbol{} & = & \left(\begin{array}{cc}
a-by & -bx\\
dy & -c+dx
\end{array}\right)
\end{array}
\]
To see that$(x^{*},y^{*})=(0,0)$, we note that the Jacobian simplifies
to:
\[
\boldsymbol{J}=\left(\begin{array}{cc}
a & 0\\
0 & -c
\end{array}\right)
\]
which has the corresponding eigenvalues  $a$ and $-c$. Thus it is
an unstable saddle point. Considering the other equilibrium, note
that when we substitute $(x^{*},y^{*})=(\frac{c}{d},\frac{a}{b})$
into the Jacobian matrix we get:
\begin{equation}
\boldsymbol{J}=\left(\begin{array}{cc}
0 & \frac{-bc}{d}\\
\frac{da}{b} & 0
\end{array}\right).
\end{equation}

Notice that $det(\boldsymbol{J})=ac$ and $tr(\boldsymbol{J})=0$.
Thus the eigenvalues associated with this matrix are purely imaginary,
meaning the fixed point is neutrally stable. As a result, we cannot
draw any conclusions on the dynamic behavior of system (3.1) by inspecting
the Jacobian. 

To study the global dynamic behavior of system (3.1), we introduce
the concept of the \emph{first integral}:
\begin{defn}
A continuously differentiable function $F:\mathbb{R}^{2}\rightarrow\mathbb{R}^{2}$
is said to be a first integral of a system $\dot{\boldsymbol{x}}=\boldsymbol{f}(\boldsymbol{x})$
where $\boldsymbol{x}\in\mathbb{R}^{2}$, if $F$ is constant for
any solution $\boldsymbol{x}(t)$ of the system. 
\end{defn}
It is important to note that if a first integral exists, it is not
unique. If $F(\boldsymbol{x})$ is a first integral, then so is $F(\boldsymbol{x})+C$, for some constant $C$
Further, the constancy of $F(\boldsymbol{x})$ is expressed as $\frac{dF(\boldsymbol{x})}{dt}=0$;
the constant expression $F(\boldsymbol{x})+C$ defines \emph{level
curves} for different values of the constant $C$. If the saddle point
is the only fixed point, the level curves are given by stable and
unstable manifolds \cite{Lorenz93}. However, if the unique fixed
point is a center, the level curves are closed orbits and any initial
point is located in a closed orbit (except the fixed points). Figure
6 illustrates this concept. For different values of $C$, the curves
$L_{1},L_{2},$ and $L_{3}$ represent different level cures and each
level curve has the property that $\frac{dF(\boldsymbol{x})}{dt}=0$. 

\begin{figure}
\begin{centering}
\includegraphics[scale=0.35]{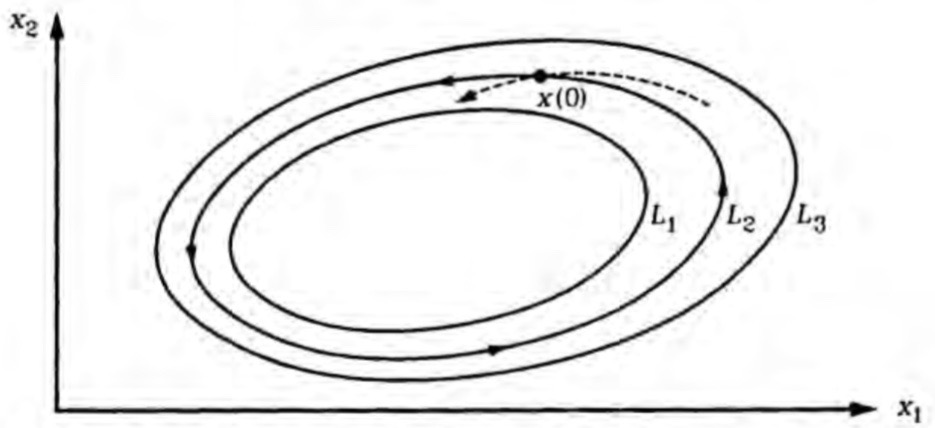}\caption{Level Curves in a System with a First Integral \cite{Lorenz93}}

\par\end{centering}

\end{figure}

\begin{prop}
When a system has a first integral, the dashed line in the diagram
above does not exist.\end{prop}
\begin{proof}
Consider the point $\boldsymbol{x}(0)$ located in $L_{2}$. Its trajectory
for $t<0$ and $t>0$, which is passing through this point, is given
by $\phi_{t}(\boldsymbol{x}(0))$. Note that the point is a point
in a level curve and can thus be described by
the constant $F(\boldsymbol{x}(0))$. The point $\phi_{t}(\boldsymbol{x}(0))$
when $t>0$ must \emph{also} be located in the level curve, else $F(\boldsymbol{x})$
would not be constant for any solution. Then it follows that $F(\boldsymbol{x}(0))=F(\phi_{t}(\boldsymbol{x}(0)))$
for all $t>0$ and $t<0$, \emph{i.e. }the trajectory indicated above
cannot exist when the system has a first integral.
\end{proof}
All initial points are located on one of the infinitely many level
curves determined by different values of $C$. 
\begin{thm}
The predator-prey system is a first integral.\end{thm}
\begin{proof}
To begin, it is helpful to eliminate the time component\emph{,} which
is done by dividing both equations:
\[
\frac{dy}{dx}=-\frac{(c-dx)y}{(a-by)x}
\]
 Note we are left with two integrable equations. After rearrangement,
dividing the equation by $xy$ and integrating, we have:

\begin{equation}
-aln(y)+by-cln(x)+dx=A
\end{equation}
 Where A is a constant. Further (3.3) can be rewritten as: 
\begin{equation}
y^{-a}e^{by}x^{-c}e^{dx}=B
\end{equation}
 if we exponentiate through. Here $B=e^{A}$. If we set (3.4) equal
to $F(x,y)$, we can see that the predator-prey system is a first
integral by differentiating with respect to time.
\[
\frac{d}{dt}F(x,y)=\frac{\partial F(x,y)}{\partial x}\dot{x}+\frac{\partial F(x,y)}{\partial y}\dot{y}.
\]
\\
The partial derivatives are 
\[
\frac{\partial F(x,y)}{\partial x}\dot{x}=F(x,y)(-\frac{c}{x}+d)
\]
\\
and 
\[
\frac{\partial F(x,y)}{\partial y}\dot{y}=F(x,y)(-\frac{a}{y}+b).
\]
 Substituting in, we have
\[
\begin{array}{ccc}
\frac{d}{dt}F(x,y) & = & F(x,y)(-\frac{c}{x}+d)(a-by)x+F(x,y)(-\frac{a}{y}+b)(-c+dx)y\\
 & = & 0
\end{array}
\]
 \\
Thus the predator prey system is a first integral. 
\end{proof}
This result, combined with an earlier theorem, yields the following
theorem. 
\begin{thm}
Every trajectory of the Lotka-Volterra equation is a closed orbit,
except for the fixed point $(x^{*},y^{*})$ and the coordinate axes. \end{thm}
\begin{proof}
It follows that closed orbits cannot be limit cycles, otherwise the
trajectories which approach limit cycles are not closed orbits. Since
each point in the phase space is located in a closed orbit, the initial
values of $(x(0),y(0))$ determine which of the infinitely many closed
orbits describe the dynamical behavior of the system.
\end{proof}
With the help of the first integral, the predator-prey system is classified
as a conservative dynamical system. A sample phase portrait for the
predator-prey model is shown in Figure 7.

\begin{figure}
\begin{centering}
\includegraphics[scale=0.3]{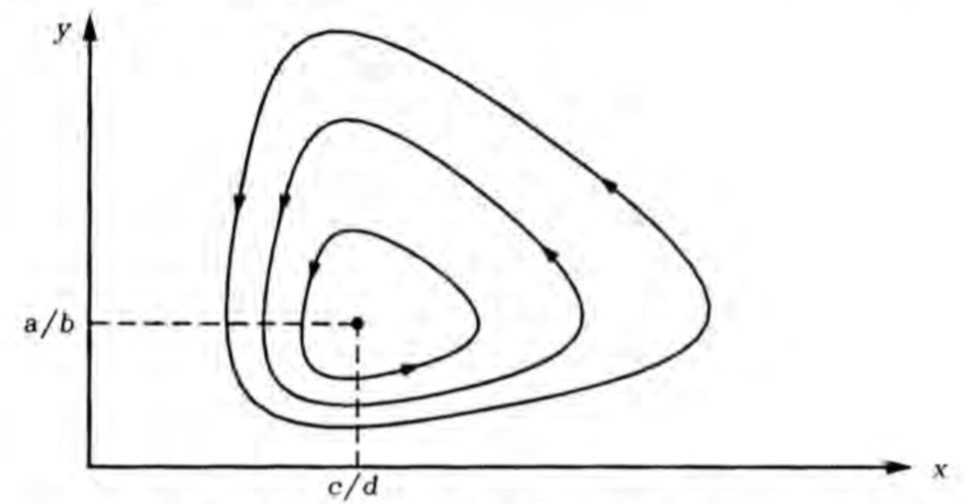}
\par\end{centering}

\caption{Sample Lotka-Volterra Phase Diagram \cite{Lorenz93}}
 
\end{figure}

\section{Goodwin's Class Struggle Model}

\subsection{History of Macroeconomic Growth Models}

The topic of modeling economic growth is one of the chief concerns
of macroeconomics. Over its theoretical development, the behavior
of economic systems organizes itself into a few main categories. The
first and earliest type concluded that markets will always obtain
a stable equilibrium; random shocks are possible and observable, but
equilibrium will eventually be restored. This didn't reflect reality
very well, though. 

It has long been known that growth happens in cycles. The first models
were descriptive in nature. That changed when Lowe \cite{Lowe26}
called for a theoretical system to describe economic cycles. Hayek
\cite{Hayek29,Hayek31} made the first attempt. He developed a theory
of economic cycles based on an interdependent equilibrium system.
This still failed to provide a full enough explanation. Keynes \cite{Keynes36}
was more successful, as he was able to develop a closed, interdependent,
and consistent theoretical structure that was able to determine phenomena
like aggregate output and unemployment. Yet, Keynes still failed to
adequately explain the business cycle, a vital concept to understand
and model in the growth of economies. 

Keynes' theoretical structure needed to be extended into a more dynamic
long-run perspective. The Oxbridge Phase of the Keynesian Revolution
provided the extension to Keynes' General Theory. The Oxbridge Phase
led to three developments in economics: 
\begin{enumerate}
\item The development of multiplier-accelerator theories of cycles.
\item The development of non-linear endogenous mechanisms.
\item The development of Keynesian cycle theory. 
\end{enumerate}
The multiplier principle implies that investment will increase output,
while the accelerator principle suggests that greater output will
increase investment. Thus a feedback cycle is developed. This idea
is attributed to Harrod \cite{Harrod36} who adopted this theory to
create cycles of growth. Harrod's theory was later mathematized by
Hicks \cite{Hicks50}. Kaldor \cite{Kaldor40} and Goodwin \cite{Goodwin51}
worked on the development of endogenous cycles. These different models,
which are non-linear in structure, use mathematics to address the
dynamics of income distribution. The structure and methods found in
Kaldor and Goodwin's work differ greatly in both the structure and
the focus of their models. Finally, the development of Keynesian cycle
theory reintroduced financial variables, moving beyond output growth
and income distribution. The two most well known models within this
realm are Keynes-Wicksell's \cite{Stein66} model on monetary policy
and Minsky's \cite{Minsky77} financial cycle theories. 

The remainder of this paper focuses on Goodwin's most famous model
of endogenous growth cycles as well, as some of the modifications to it.
In 1967 Goodwin \cite{Goodwin67} introduced a simplistic model about
the dynamics of wages and employment. His model is analogous to the
Lotka-Volterra predator-prey model, where wages correspond to predators
and employment corresponds to prey. A brief explanation yields some
intuition to why this model makes sense. 

In this economy there are two groups: workers and employers; each
has some bargaining power. At high levels of employment, the bargaining
power of workers is high. They are able to drive up wages,f thus reducing
profits. As profit levels fall, fewer workers are able to be hired,
yielding lower levels of employment. This low level of employment
means lower wages and therefore higher profits. With higher profits,
more workers are hired and the employment level rises again. A cyclic
pattern emerges.

Goodwin's model originates from an idea proposed by Karl Marx in \emph{Capital}
\cite{Marx85}. Marx argued ``capitalism's alternate ups and downs
can be explained by the dynamic interaction of profits, wages, and
employment.'' The growth of profits fuels production and thereby
labor demand; wages rise and shrink profits, which erode the basis
for accelerated accumulation. Further, Marx argues that the employment
rate and worker's share of income triggers these cycles. Goodwin formalized
this concept with mathematics, which is done through the framework
of the predator-prey model.

\begin{figure}
\centering{}\includegraphics[angle=270,scale=0.115]{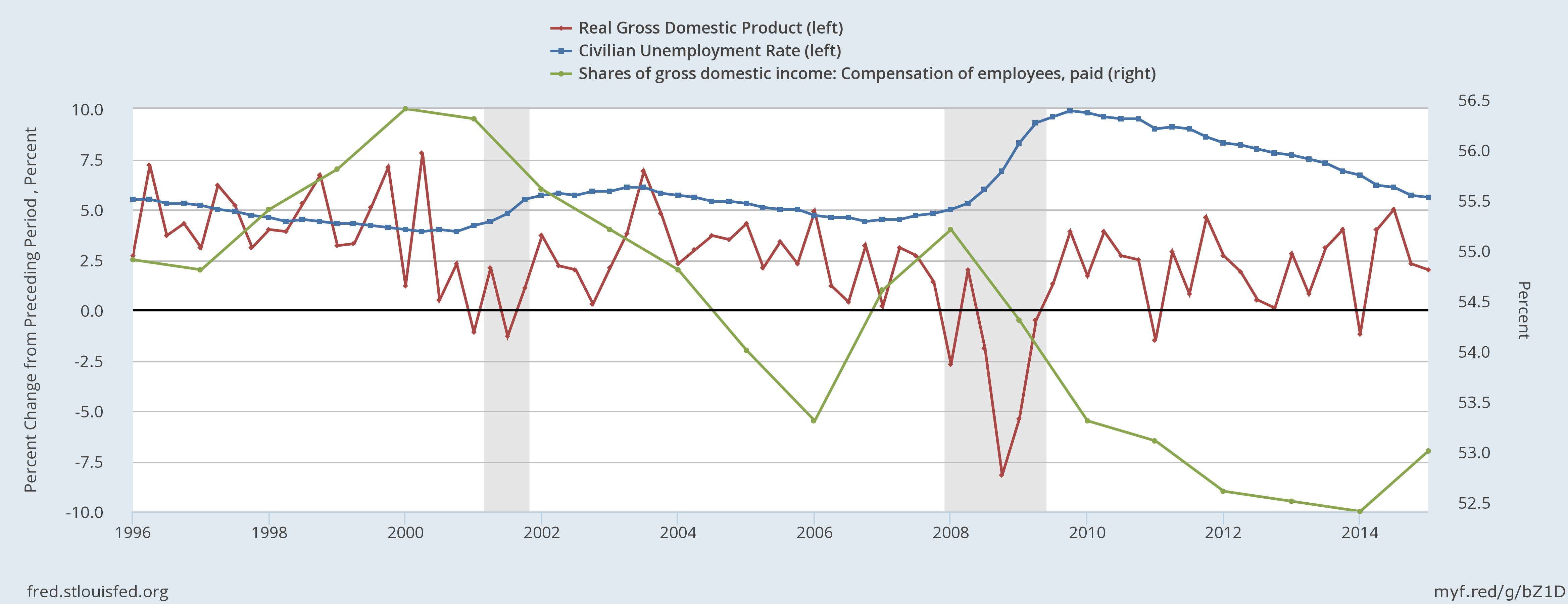}\caption{Real GDP Growth, Unemployment Rate, \& Worker's Share of Domestic
Income (US 1996-2015)}
\end{figure}

These cycles are not to be confused with the business cycle which
are also apparent in economic growth. However, the business cycle
\emph{can} affect wage-employment cycles, and fluctuations in wage-employment
cycles can affect business cycles, leading to recessions. Figure 8
shows output growth, unemployment and workers' share (net wages as
a share of national income) in the United States from 1996 to 2015.
A somewhat cyclic pattern is apparent. High levels of output growth,
a high rate of employment (low rate of unemployment), and increases
in real wages are tied with the economic expansion of the late 1990s
In the early 2000s, output growth slowed substantially, real wages
fell, and unemployment slightly increased. From 2006-2008, there was
strong economic growth, which was again mirrored by low unemployment
and increases in worker's share of national income. With the onset
of the great recession in 2008, output growth again slowed substantially,
unemployment rose drastically, and worker's share of national income
plummeted. It is important to observe that Goodwin's cycle does not
directly deal with business cycles, since changes in employment and
wages are dynamically modeled. However, Goodwin cycles can coexist
with business cycles. 

Goodwin's model found attention among political economists. It's not
entirely clear if it's due to the oscillatory properties it predicts,
the suggested analogy between predator-prey interdependence and worker's
class struggle, or some combination of the two. Nevertheless, Goodwin's
Model serves as a guide to show the power of mathematics in economic
modeling.

\subsection{Deriving the Model}

We begin with the following assumptions. The theoretical
economy being considered is a \emph{closed }economy, simply meaning
that there is no international trade.\\
\\
\textbf{(A1)} Technical Progress grows at a constant rate.\\
\textbf{(A2)} The labor force grows at a constant rate.\\
\textbf{(A3)} There exist only two homogeneous and non-specific factors
of production: capital and labor.\\
\textbf{(A4)} All quantities are real and net.\\
\textbf{(A5)} All wages are consumed; all profits are saved and invested.\\
\textbf{(A6)} There is a constant capital-output ratio.\\
\textbf{(A7)} A real wage that rises in the neighborhood of full employment,
expressed by the Phillips Curve.\\
\\
Further, the following list of abbreviations, definitions, and relations
describes the framework of the economy. 
\begin{itemize}
\item Output: $Y$
\item National Income: $q$
\item Labor: $L$ 
\item Capital: $k$
\item Wage Rate: $w$
\item Labor Productivity: $a$
\item Labor Income: $wL$
\item Labor Income Share: $u$
\item Profit Income Share: $\pi$
\item Savings: $(1-\frac{w}{a})Y$
\item Capital Output Ratio: $\sigma$
\item Labor Supply: $N$
\item Employment Rate: $v$
\end{itemize}
We begin with the tautological equation
\begin{equation}
q=\frac{Y}{L}\cdot\frac{L}{N}\cdot N
\end{equation}
 where $\frac{Y}{L}$ is defined to be labor productivity and $\frac{L}{N}$
is defined to be the employment rate. Thus the equation simplifies
to:
\begin{equation}
q=a\cdot v\cdot N.
\end{equation}
 We wish to take the logarithmic derivative with respect to time.
\footnote{With regard to the construction of our model, taking the logarithmic
derivative allows us to understand how variables in our system change
over time. Recall that if 
\[
f=\frac{g}{h}
\]
then by the quotient rule, the percent change over time is:
\begin{align*}
\frac{\dot{f}}{f}= & \frac{1}{f}\left[\frac{d}{dt}\left(\frac{g}{h}\right)\right]\\
= & \frac{1}{f}\frac{\dot{g}h-\dot{h}g}{h^{2}}\\
= & \frac{h}{g}\left(\frac{\dot{g}}{h}-\frac{g\dot{h}}{h^{2}}\right)\\
= & \frac{\dot{g}}{g}-\frac{\dot{h}}{h}.
\end{align*}}
 Thus, taking the logarithmic derivative of (4.2) yields
\begin{equation}
\frac{\dot{q}}{q}=\frac{\dot{a}}{a}+\frac{\dot{v}}{v}+\frac{\dot{N}}{N}
\end{equation}
and simplifies to
\begin{equation}
\frac{\dot{q}}{q}=\alpha+\frac{\dot{v}}{v}+\beta.
\end{equation}
 Here, $\alpha$ is the increase in productivity and $\beta$ is the
growth rate in labor supply. Recall from our assumptions that both
$\alpha$ and $\beta$ are constants. 

The national income is distributed 100 percent between capitalist's
share of national income (\emph{$\pi$)} and worker's share of national
income ($v$), which can be mathematically expressed by $1=u+\pi$.
Further, by assumption, since all profits are reinvested, it follows
that 
\[
\dot{k}=q\cdot\pi
\]
\[
\frac{\dot{k}}{q}=1-u.
\]
 $\sigma$ represents a constant capital output ratio. Mathematically,
then, $\sigma=\frac{k}{q}$. Diving by $\sigma$ yields:
\begin{align*}
\frac{\dot{k}}{\sigma}= & \frac{(1-u)}{\sigma}\cdot q\\
\frac{q\cdot\dot{k}}{k}= & \frac{(1-u)}{k}q^{2}\\
\frac{\dot{k}}{k}= & \frac{(1-u)}{k}q\\
\frac{\dot{k}}{k}= & \frac{(1-u)}{\sigma}.
\end{align*}
We note some observations which will help us simplify our equation.
Note first that the equation above states that the growth rate of
capital is dependent on worker's share of national income. As a result,
we can replace the growth rate of capital with the growth rate of
national income. Similarly, since we know that $\sigma$ is constant,
the growth rate of national income must be the same as the growth
rate of the capital stock. Thus with these two assertions, a fluctuation
in capital leads directly to a fluctuation in the national income:
\begin{equation}
\frac{\dot{k}}{k}=\frac{\dot{q}}{q}=\frac{1-u}{\sigma}.
\end{equation}
Setting (4.4) equal to (4.5) and solving for $\dot{v}$, we get the
first differential equation describing employment rate:
\begin{equation}
\dot{v}=v(t)\left(\frac{1}{\sigma}-(\alpha+\beta)-\frac{u(t)}{\sigma}\right).
\end{equation}
Compared to the Lotka-Volterra model outlined above, (3.1) is equivalent
to the function describing prey. 

The second equation is established in a similar manner. To assist
in defining it, we begin with the Phillips Curve. Phillips \cite{Phillips58}
attempted to estimate a correlation between changes in wages and
unemployment rate. Goodwin took advantage of this relationship, however
for simplicity, he proposed a linearized version. While the exact
relationship is not known, this linearized version will suffice to
model the relationship. To transform the unemployment rate into the
employment rate, we begin by noting that the labor supply in our economy
is composed fully of employed and unemployed workers, \emph{i.e.}
$1=v+z$ where $z$ is defined to be the unemployment rate. Figure
9 graphically notes the transformation of the Phillips curve from
an unemployment rate to an employment rate. 

\begin{figure}

\begin{centering}
\includegraphics[scale=0.5]{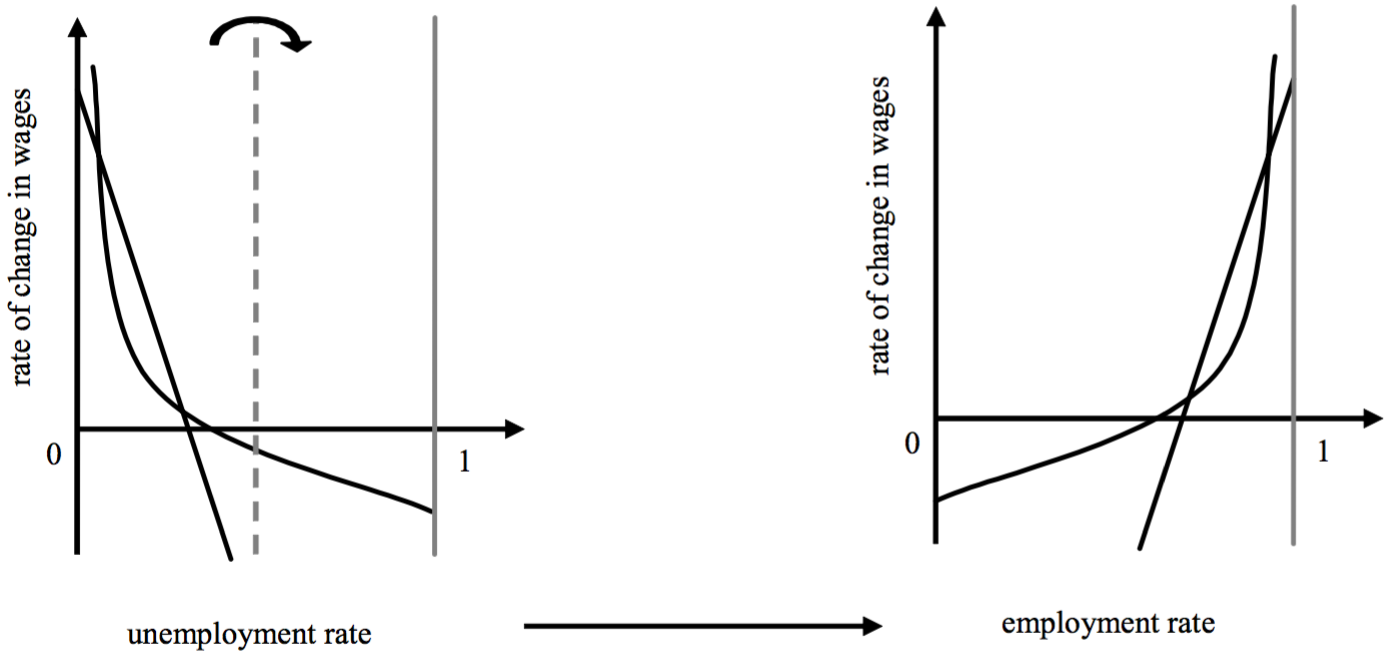}\caption{Transformation of the Phillips Curve}

\par\end{centering}

\end{figure}

With the transformed and linearized Phillips curve, we note
\begin{equation}
\frac{\dot{w}}{w}=-\gamma+\rho v.
\end{equation}
Here, $\gamma$ is defined to be the lines intersection with the $y$-axis
and $\rho$ is the slope of the line. We define worker's share of
national income as 
\begin{equation}
u=\frac{w}{a}.
\end{equation}
A logarithmic differentiation with respect to time yields the growth
rate of worker's share of national income
\[
\frac{\dot{u}}{u}=\frac{\dot{w}}{w}-\frac{\dot{a}}{a}
\]
\begin{equation}
\frac{\dot{u}}{u}=\frac{\dot{w}}{w}-\alpha.
\end{equation}
 Finally, plugging in (4.7) into (4.9) and solving for $u$ yields
the following equations for change in worker's share of national income:
\begin{equation}
\dot{u}=u(t)\left(-(\alpha+\gamma)+(\rho v(t))\right)
\end{equation}
The differential equation describing the dynamics of worker's share
of national income is analogous to the equation describing the predator
population in the Lotka-Volterra model.

To summarize, we have a system of differential equations to describe
the Goodwin Model:

\begin{align}
\dot{v}= & v(t)\left(\frac{1}{\sigma}-(\alpha+\beta)-\frac{u(t)}{\sigma}\right)\label{eq:2.17}\\
\dot{u}= & u(t)\left(-(\alpha+\gamma)+(\rho v(t))\right).\nonumber 
\end{align}
This is simply a version of the predator-prey model, where $a=\frac{-(\alpha+\beta)}{\sigma}$,
$b=\frac{1}{\sigma}$, $c=\alpha+\gamma$ and $d=\rho v$. Economically
speaking the first summand in the parenthesis represents the natural
growth rate of each variable, while the second summand gives the density.
When there is no employment, worker's share of national income tends
to zero; when worker's share of income tends to zero, the employment
rate increases since no relative labor costs exist. 

As with the predator-prey model, (4.11) has two fixed points: the
trivial fixed point at the origin and the non-trivial fixed point.
Since the trivial fixed point makes little economic sense, it is ignored.
The non-trivial fixed point is
\begin{align*}
v^{*}= & \frac{\alpha+\gamma}{\rho}\\
u^{*}= & 1-\sigma(\alpha+\beta).
\end{align*}

The Jacobian evaluated at the non-trivial fixed points is
\begin{equation}
\boldsymbol{J}=\left(\begin{array}{cc}
0 & -\frac{\alpha+\gamma}{\sigma\rho}\\
\rho(1-\sigma(\alpha+\beta)) & 0
\end{array}\right).
\end{equation}
As the system (4.11) and the Jacobian (4.12) are structurally identical
to the Lotka-Volterra equations, the Goodwin model is a conservative
dynamical system. Thus every initial point in the model is located
in a closed orbit. The behavior is like that of the general predator-prey
model, and is shown in Figure 10. There are four quadrants regarding
the fixed point $(u^{*},v^{*})$. The small arrows indicate the behavior
of worker's share of national income and the employment rate. It can
be described as follows: when labor share of national income ($u$)
is greater than $u^{*}$ and the employment ratio ($v$) is \emph{also
}greater than $v^{*}$, the main economic variables considered in
the Goodwin model all fall down because the pressure put on capitalist's
profits weakens the investment activities. The wage continues to rise
until $v$ becomes less than or equal to $v^{*}$. 

\begin{figure}
\begin{centering}
\includegraphics[scale=0.7]{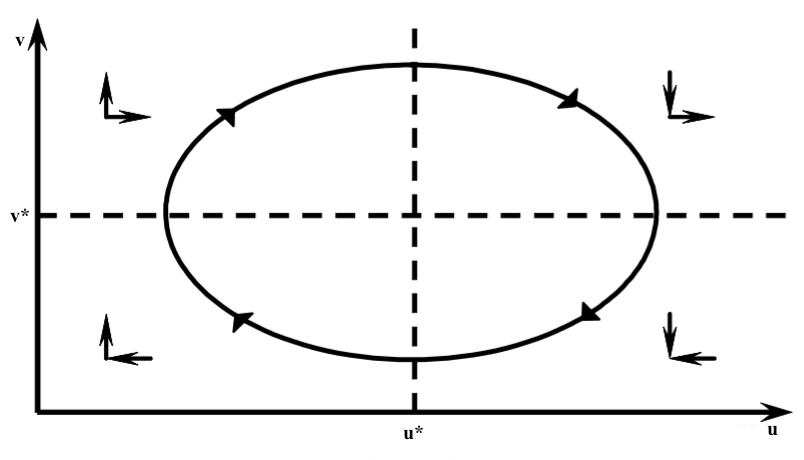}\caption{Behavior of Worker's Share and Employment Rate in Different Sectors}

\par\end{centering}

\end{figure}

The period of the cycles is like that in the Lotka-Volterra model
with respect to the variables:
\[
T=\frac{2\pi}{\sqrt{(\alpha+\gamma)\cdot(\frac{1}{\sigma}-(\alpha+\beta)}}
\]

Note the difference between the fixed points in the Goodwin model
as compared to the Lotka-Volterra model. The increase in productivity,
$\alpha$, is implemented into both differential equations. To achieve
a closed orbit around a defined coordinate $(u^{*},v^{*})$, the values
of the variables must be in a specific ratio to each other. 

The dynamics of this model lend theoretical credibility to Goodwin's
and thereby Marx's belief that a capitalist economy is permanently
oscillating. It is important to note, in an economic context, that
the analogy is superficial, as it does not refer directly to the \emph{functional}
income shares of capitalists and their workers or even to their population
size.

\subsection{Updated Goodwin Models}

As with the Lotka-Volterra model, the Goodwin model has been criticized
as too simplistic; namely that the model is composed under an isolated
set of assumptions which most likely do not reflect reality greatly.
The model is simple and exhibits oscillations, as shown. However,
the property of structural instability limits its applications. Many
modifications have been made to the model to more accurately model
changes in employment rate and worker's share of national income.
A few of these modifications, as well as their dynamical behavior,
are discussed below. 

The Lotka-Volterra model, and thus the original Goodwin model, are
dynamical systems whose behavior is very sensitive to small alterations
in their functional structure. We call dynamical systems of this sort
\emph{structurally unstable}. A number of the updates to the Goodwin
model focus on the structural instability that is found in all conservative
dynamical systems. To demonstrate the effect of small perturbations,
we consider an arbitrary but simple modification to the Phillips curve.
Before we assumed that the changes in wages were dependent on \emph{only}
 the employment rate. Now assume that they are dependent on employment
rate and worker's share of national income, a scenario that is a bit
more realistic:
\[
\frac{\dot{w}}{w}=f(v)+g(u).
\]
We assume that worker's share increase if workers are at a disadvantage
in the functional income distribution, which is reflected as $g(u)>0\mbox{ }\forall u$
and $g'(u)<0$. 

With this modified Phillips curve being considered, the updated Goodwin
Model leads equations are as follows: 
\begin{align}
\dot{v}= & v(t)\left(\frac{1}{\sigma}-(\alpha+\beta)-\frac{u(t)}{\sigma}\right)\label{eq:2.20}\\
\dot{u}= & u(t)\left(-(\alpha+\gamma)+(\rho v)+g(u)\right)\nonumber 
\end{align}
with the non new trivial fixed points being 
\begin{align*}
v^{*}= & \frac{\alpha+\gamma-g(u)}{\rho}\\
= & \frac{\alpha+\gamma-(\frac{1}{\sigma}-(\alpha+\beta))}{\rho}\\
u^{*}= & 1-\sigma(\alpha+\beta).
\end{align*}
The Jacobian evaluated at these fixed points is
\[
\boldsymbol{J}=\left(\begin{array}{cc}
0 & \frac{g(u^{*})-\alpha+\gamma}{\sigma\rho}\\
\rho(1-\sigma(\alpha+\beta)) & g'(u^{*})(\frac{1}{\sigma}-(\alpha+\beta))
\end{array}\right).
\]
Unlike before, the determinant of $\boldsymbol{J}$ is not always
positive. Suppose that $g'(u^{*})>0$ is sufficiently small such that
$\boldsymbol{J}$ is in fact positive. The trace of $\boldsymbol{J}$
will be different from zero even for a small value of $g(u^{*}),\mbox{ }u^{*}\neq0$.
By assumption of $g'(u)<0$, the trace is also negative. Hence, the
real parts of the complex conjugate eigenvalues are negative and the
fixed point $(v^{*},u^{*})$ is locally asymptotically stable. Thus
system (4.13) possess an attractor and is no longer a conservative
system but rather a dissipative system. 

Similar modifications, with an extended Phillips curve $f(u,v)$,
done by Cugno and Montrucchio \cite{Cugno82}, were able to provide
global stability results. Other modifications can also be constructed
which yield similar results. Samuelson \cite{Samuelson71,Samuelson72}
showed that when considering diminishing returns in the Lotka-Volterra
framework, it can change the classical behavior of a conservative
system. While Samuelson did not consider the Goodwin model \emph{directly
}- he commented directly on the general Lotka-Volterra model - diminishing
or increasing returns to scale are considered by assuming the capital-output
ratio changes with output, rather than being held constant. In fact,
the addition of a term that influences the growth rate of some variable
\emph{and} which depends on on the value of this variable is equivalent
to the introduction of a dampening effect. The conservative framework
of the original Goodwin equations will change and result in a dissipative
dynamical system. 

Increasing the dimension of the original model by considering additional
state variables is another way to modify Goodwin's work. Additional
state variables may include capital-output ratios, or non-constant
growth rates of the labor supply and labor productivity. It is also
possible to increase the dimension by introducing a lag-structure,
which is also more reflective of economic reality. Vadasz \cite{Vadasz07}
considers a number of these variations and implements them into one
cohesive model. 

Recall that in the absence of predators, prey grow exponentially.
Goodwin's original model assumed this. Economically, this means that
in the absence of wages, employment grows at an exponential rate and
quickly surpasses full employment, growing without bound. In reality,
however, the average labor share - defined as the ratio of employed
workers to the working age population (people between 15-65) - is
steadily below 75 percent. A number of problems arise when this type
of growth is considered. First, employment cannot increase without
limits and setbacks to productivity gains. In fact, diminishing returns
to productivity are typically found. Further, workers are not homogenous
and all do not have the same level of productivity. For example, during
the periods where labor is being shed (employment rate is falling),
the less trained and less productive workers will be the first to
go. Similarly in the opposite case, when employment rates are increasing,
the most productive will be hired first, though whatever is available
in the labor market will also be hired, meaning less knowledgeable
and skilled workers. 

To make the model more representative of reality, Vadasz considers
in the absence of wages, employment growing according to a logistic
growth function, rather than an exponential growth function
\[
\dot{v}=(\frac{1}{\sigma}-(\alpha+\beta))\frac{(1-v)}{K}v
\]
where $K$ represents the carrying capacity, or an upper bound on
employment growth. In his model, Vadasz lets $K=1$ which economically
translates into labor \emph{not} being able to surpass the total population
in a closed economy.

Similarly, Vadasz considers the reaction of labor share of national
income to employment and whether those two variables move in tandem.
Changes in employment do not have an instantaneous effect on the labor's
share. Rather, labor's share, reflected in wages, are sticky since
they are determined in advance by contracts and rarely do those contracts
take into effect future changes in demand for labor. In a recession,
for example, wage rates react sluggishly to growing unemployment.
The lowest wage rate usually is attained when the economy is already
growing again. The delay can be modeled by substituting in a weight
function to equation 2 of system (4.11). We replace $v$ with $z$
\[
z=\int_{0}^{t}u(\tau)G(t-\tau)d\tau
\]
where $G$ is a nonnegative integrable weight function with the following
property:
\[
\int_{-\infty}^{t}G(t-\tau)d\tau=\int_{0}^{\infty}G(s)ds=1.
\]
The growth rate of worker's share of national income depends on the
employment rate in the past. The way in which the growth rate depends
on past values is determined by the choice of weight function. The
choice of weight function further looks to remove the structural instability
of original Goodwin model. Consider the weight function to be 
\[
G(s)=G_{1}(s)=ae^{-as},\mbox{ }a>0.
\]
Here a $\frac{1}{a}$ discount reaction of worker's share of national
income is considered. Employers take into consideration the changes
in capital when setting wage contracts. Recall that, by assumption,
the capital ratio is fixed. Since this is the case, and output is
directly related to employment, capital depreciation will cause employers
to discount past employment levels when setting wages by a discount
rate $a$. Further, past employment levels that happen further away
will have a smaller effect on labor's share of national income, and
will continue to decrease exponentially as time goes on. 

We then have an extended Goodwin model where
\begin{align*}
\dot{v}= & v(t)\left((\frac{1}{\sigma}-(\alpha+\beta))\frac{(1-v)}{K}-\frac{u}{\sigma}\right)\\
\dot{u}= & -u(t)\cdot(\alpha+\gamma)+(\rho z)\\
= & -u(t)\cdot(\alpha+\gamma)+(\rho\int_{0}^{t}u(\tau)ae^{-a(t-\tau)}d\tau).
\end{align*}
Note that differentiating $z$ with respect to time yields the ordinary
differential equation 
\[
\dot{z}=a(v-z)
\]
 and thus with the previously derived equations, the Goodwin model
turns into a three-dimensional system
\begin{align*}
\dot{v}= & v(t)\left((\frac{1}{\sigma}-(\alpha+\beta))\frac{(1-v)}{K}-\frac{u}{\sigma}\right)\\
\dot{u}= & -u(t)\cdot(\alpha+\gamma)+(\rho\int_{0}^{t}u(\tau)ae^{-a(t-\tau)}d\tau)\\
\dot{z}= & a(v-z).
\end{align*}
The system has a fairly straightforward economic interpretation. In
the labor market, labor's share of national income will not be set
by actual employment $v$ but according to expectations of future
employment levels, which are based on past employment levels. These
expectations will continuously change and correct themselves as expressed
by $z$.

The system can be shown to have three fixed points. The trivial and
unstable fixed points $(0,0,0)$ and $(1,0,1)$, which represent the
absence of wages. The nontrivial fixed point is 
\begin{align*}
v^{*}= & \frac{\alpha+\gamma}{\rho}\\
u^{*}= & (\rho-(\alpha+\gamma))\cdot\sigma(1-\sigma(\alpha+\beta))\\
z^{*}= & \frac{\alpha+\gamma}{\rho}.
\end{align*}
The conservative character of the original two-dimensional Goodwin
model has disappeared through the introduction of an exponential lag
structure. As a result, a dissipative system has emerged. 

Since the Goodwin model is sensitive to small perturbations and suffers
from structural instabilities, as soon as a dissipative structure
prevails, this modified Goodwin system can exhibit converging or diverging
oscillations and limit cycles depending on the assumed dampening or
forcing terms present. These modifications still allow for oscillating
behavior of the economically most relevant variables, like labor share
of national income and employment rate. Further modifications to the
model include relaxing the assumption that the economy being modeled
is closed and considering an open economy with international trade,
which will be explored further later in the paper.

\section{Sheaf Methods for Analyzing Complex Systems}

\subsection{Motivation}

Complex dynamical systems, including the ones discussed above, can
be difficult to construct to accurately model specific phenomena.
They are even harder to study and analyze. We seek to unlock a set
of tools to help in this construction and analysis of these complex
systems. Sheaf theory provides an avenue to do so, lending a hand
in constructing and analyzing models that are described by a system
of equations. Sheaves, which will be formally defined later, are a
way to locally track data and synthesize these local bits into a consistent
whole. Complex multi-model systems, then, are easily stitched together
from diagrams of smaller models. This notion can be abstracted into
a more general framework, which is done by Robinson \cite{Robinson16}. 

In general, there is an interaction between the individual models
and how they interact together. To abstract to its most general setting,
the models consist of spaces and maps between them; thus it is natural
to consider the system's topology. By modeling the topology of the
system \emph{first} and then the spaces and maps of the individual
models that are specified by the system, this construction and procedure
naturally leads to sheaves. When dealing with sheaves, it is important
to note which type of sheaf we are analyzing \emph{and }a kind of
topological space. The theory going forward is built over partial
orders, which leads to a balance of expressivity and computational
ability that other space types do not readily lend to. The multi-modal
systems considered here are sheaves of smooth finite dimensional manifolds
on posets.

The construction of sheaves is the most fundamental way to express
the topological relationships between the variables and equations.
It also yields a diagrammatic intuition for how variables and equations
are all simultaneously related. Further, with the sheaf model, any
analysis done on the sheaf is equivalent to analysis done on the system
of equations, allowing us to study and analyze locally and globally
stable states of the system. 

In summary, the sheaf construction makes apparent and possible two
capabilities:
\begin{enumerate}
\item It allows one to combine seemingly different dynamical models into
a multi-model system in a fundamental and consistent manner.
\item The sheaf diagram allows for analysis of local and global sections
to be done in an easier manner than would be on the system of equations.
\end{enumerate}

\subsection{Goodwin Model as a Sheaf}

Recall the Lotka-Volterra model:
\begin{eqnarray*}
\frac{dv}{dt}= & \dot{v}= & v(t)\left(\frac{1}{\sigma}-(\alpha+\beta)-\frac{u(t)}{\sigma}\right)\\
\frac{du}{dt}= & \dot{u}= & u(t)\left(-(\alpha+\gamma)+(\rho v(t))\right).
\end{eqnarray*}

The dynamical model presented here involves a collection of two state
variables ($u$ and $v$) and two equations ($\dot{u}$ and $\dot{v}$),
which are both functions of time. For both equations, the values of
$u$ and $v$ determine all future values for the equation. As discussed
above, the solution exhibits interesting oscillatory behavior. One
way to gain an understanding in the behavior of the solutions is to
build out some visual representations of the system. These diagrams
will highlight on the nuances of the causal relationships between
the different state variables. 

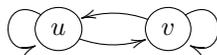
\begin{figure}
\[
\xymatrix{*++[o][F]{u}\ar@/_{.5pc}/[r]\ar@(ul,dl) & *++[o][F]{v}\ar@/_{.5pc}/[l]\ar@(ur,dr)}
\]
\caption{Variable Dependency Graph}

\end{figure}
Figure 11 describes the most basic dependency relationship between
the state variables. Here an arrow from one variable to the next -
\emph{i.e. $u\rightarrow v$ }- states that the future value of $v$
is partially dependent on the current value of $u$. This dependency
diagram, however, isn't entirely representative of what the equations
are describing, since there is no notion of a derivative present.
Now, if we include the derivatives of the state variables \emph{as
}state variables, we construct a larger diagram which provides a bit
more information about how the derivatives are determined by the values
of the state variables. Yet, it is not complete. It fails to fully
describe the relationship between the derivative and the state variable,
since, for example, $\frac{du}{dt}$ is determined both by the values
of $u$ (alone) and by $u$ and $v$ through the first equation in
our system. Our next step is for the encoding of this model to contain
all of this information into the diagram. 

\begin{figure}

\[
\xymatrix{ & *++[o][F]{u}\ar@/_{1pc}/[dl]^{\mbox{}}\ar@/^{1pc}/[dr]\\
*++[o][F]{\frac{du}{dt}} &  & *++[o][F]{\frac{dv}{dt}}\\
 & *++[o][F]{v}\ar@/^{1pc}/[ul]\ar@/_{1pc}/[ur]
}
\]
\caption{Expanded Dependency Graph}

\end{figure}
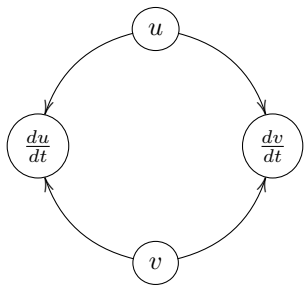

This encoding can be done simply by reinterpreting what the arrows
mean in our dependency diagram. Previously we interpreted the arrow
to mean that the variable at the head of the arrow is dependent on
the variable at the tail of the arrow. If we interpret the arrow as
a functional relationship \emph{rather} than a simple dependency,
this added information is encoded. This stronger requirement
is not present in either Figure 11 or 12. It is fairly clear why it
is not present in the Figure 11; a functional dependence is not present
in Figure 12 since, from how our system is defined, the formula $\frac{du}{dt}$
is dependent on both $u$ and $v$. To define the functional dependence
between $u$, $v$, and $\frac{du}{dt}$, it needs to come from the
pair $(u,v)$. Performing this transformation to the dependency diagram,
we obtain the following figures. The next two figures are composed
of two different types of functional dependency graphs; Figure 13
shows the functional dependencies between the state variables and
their derivatives according to variable names, while Figure 14 shows
the same relationship according to the spaces of values involved.

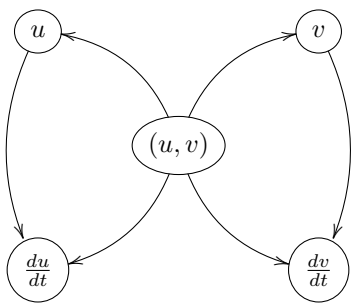
\begin{figure}
\[
\xymatrix{*++[o][F]{u}\ar@/_{1pc}/[dd] &  & *++[o][F]{v}\ar@/^{1pc}/[dd]\\
 & *++[o][F]{(u,v)}\ar@/_{1pc}/[dr]\ar@/^{1pc}/[dl]\ar@/^{1pc}/[ur]\ar@/_{1pc}/[ul]\\
*++[o][F]{\frac{du}{dt}} &  & *++[o][F]{\frac{dv}{dt}}
}
\]

\caption{Functional Dependencies - Variable Names}
\end{figure}

\begin{figure}
\[
\xymatrix{*+++[o][F]{C^{1}(\mathbb{R},\mathbb{R})}\ar@/_{1pc}/[dd] &  & *+++[o][F]{C^{1}(\mathbb{R},\mathbb{R})}\ar@/^{1pc}/[dd]\\
 & *++++[o][F]{C^{1}(\mathbb{R},\mathbb{R}^{2})}\ar@/_{1pc}/[dr]\ar@/^{1pc}/[dl]\ar@/^{1pc}/[ur]\ar@/_{1pc}/[ul]\\
*+++[o][F]{C^{0}(\mathbb{R},\mathbb{R})} &  & *+++[o][F]{C^{0}(\mathbb{R},\mathbb{R})}
}
\]

\caption{Functional Dependencies - Variable Spaces}
\end{figure}
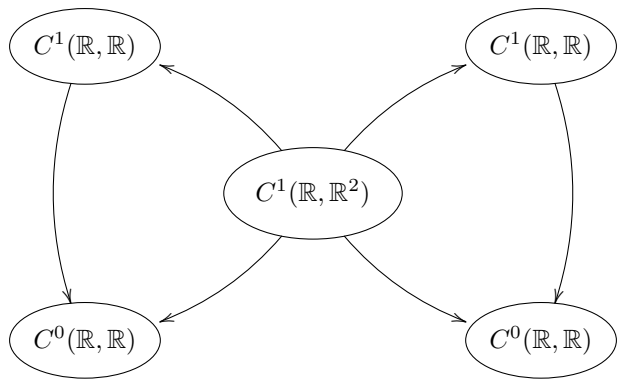

There are a number of advantages to the functional dependency diagram
described above. A useful property of their structure is that sheaves
are path independent, \emph{i.e. }the diagram commutes. Second, the
arrows from the variables or spaces are actual functions, and if one
so desires, they can be labeled as such. The arrows which correspond
to the pairs of variables to an individual variable (\emph{i.e.} $(u,v)\rightarrow u$)
are projections, while the others are are determined the definition
of derivative, and by the equation of the system. A third advantage
to this diagram is that all information from our system is captured
in the dependency diagram \emph{and} the equations can be recovered
from the diagram. Analysis that is applied to the diagram, such as
finding and studying global and local sections, is equivalent to finding
those in our system of equations. 

It is useful to look at the number of functional equations that constrain
a variable value. This is defined as the \emph{in-degree} of a variable.
Our functional dependence diagram easily allows us to identify the
in-degree values of variables in our system. Independent variables
are those which have no arrows pointing into them, \emph{i.e. }the
pair $(u,v)$. However, simply because they are independent variables
doesn't mean that they have no constraints, just that there are no
functional dependencies. Constraints, on the other hand, arise by
requiring each value to take on only one value. For example, if a
variable is determined by two functional equations, as is the case
in the Goodwin model, the independent variables in those two equations
must be chosen compatibly. There are two other possibilities for the
variables: they are completely dependent or intermediate variables.
Both $\frac{du}{dt}$ and $\frac{dv}{dt}$ are examples of dependent
variables since there exist no arrows going out of them, while the
variables $u$ and $v$ are intermediate variables since arrows go
in and out of them. 

Figure 13 is that of a partially ordered set. The advantage of this
method is the partial order ranks the variables in our system according
to their independence of each other. The arrows in the diagram point
from lower variables to higher variables in our partial order. Thus
we can think of the most independent variables as the minimal elements
of the partial order and the completely dependent variables as the
maximal elements of the partial order. 

It is easy to see Figure 14 has the same diagrammatic structure as
the partial order. However, the labeling is different. This mathematical
representation presented here is the sheaf of the Goodwin model. Recall
that the sheaf is a way to represent local consistency relationships.
As will be outlined in the next subsection, sheaves are equipped with
a number of useful properties which yield descriptive power for systems
of equations. Further, simply observing the structure of the dependency
diagram is generally illuminating in that it helps us understand the
sometimes complex relationships among variables.

\subsection{Mathematical Construction: Sheaves on Posets}

According to the analysis outlined by Robinson, topological spaces
in their full generality admit some properties that are not reflected
in practical models and thus needs constraints. Other topological
spaces aside from partially ordered sets vary in expressivity. Cell
complexes, locally finite topological spaces, abstract simplicial
complexes, and partial orders are most useful for modeling systems.
Partial orders, however, appear to be the most useful as each computational
example can be expressed with them\cite{Robinson16}. What follows,
therefore, is a description of sheaves on partial orders, since they
are the primary mathematical tool which will be used going forward.
For computational ease, we will be dealing with locally finite posets. 
\begin{defn}
A \emph{partial order} on a set $P$ is a relation $\leq$ on the
set that is 

1. Reflexive: $x\leq x$ $\forall x\in P$ 

2. Antisymmetric: If $x\leq y$ and $y\leq x$, then $x=y$

3. Transitive: If $x\leq y$ and $y\leq z$, then $x\leq z$.
\end{defn}
A pair $(P,\leq)$ is called a \emph{partially ordered set} (or \emph{poset}
for short). Typically when context is clear, we will denote $P=(P,\leq)$.
Further, a poset is called \emph{locally finite} if the set $\{z\in P:x\leq z\leq y\}$
is finite, given every pair $x,y\in P$. 

\begin{figure}

\begin{centering}
$\xymatrix{ & {c}\\
{a}\ar[ur] &  & {b}\ar[ul]\\
 & {d}\ar[ul]\ar[ur]
}
$$\xymatrix{ & {\mathscr{S}(c)}\\
{\mathscr{S}(a)}\ar[ur] &  & {\mathscr{S}(b)}\ar[ul]\\
 & {\mathscr{S}(d)}\ar[ul]\ar[ur]
}
$\caption{A poset $P$ (left) and a sheaf over $P$ (right) (Definition 5.3)
\cite{Robinson16}}

\par\end{centering}

\end{figure}
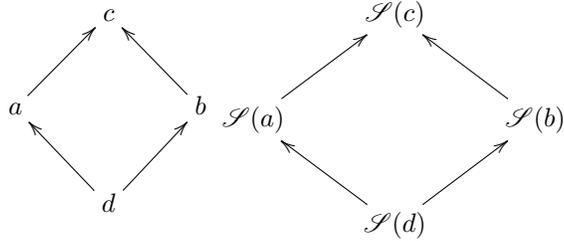

Figure 15 shows a number of diagrams. The digram to the left is that
of a poset $P$ with four elements. The poset is to be interpreted
as follows: $d\leq a\leq c$ in $P.$ The figure also includes a diagram
for a sheaf. The sheaf can be explained in terms of the diagram of
the poset. In the poset, each vertex represents an element and each
arrow points from lesser elements to greater ones. By replacing each
vertex and arrow by a set (or space) and a function respectively,
such that the composition is path independent (it does not matter
which path is taken to get from $\mathscr{S}(d)$ to $\mathscr{\mathscr{S}}(c)$),
we have created a sheaf. If all of the functions' inputs are at the
tail of the arrow, then the diagram is that of a sheaf on the Alexandroff
topology (denoted by $\mathscr{S}(\cdot)$). 
\begin{defn}
\cite{Robinson16} In a poset $(P,\leq)$, the collection of sets
of the form 
\[
U_{x}=\{y\in P:x\leq y\}
\]
for each $x\in P$ forms a base for a topology, called the \emph{Alexandroff
}topology, shown in Figure 15.
\end{defn}
This topology is particularly important as it can be built from a
partial order. 
\begin{defn}
\cite{Robinson16} Suppose that $P=(P,\leq)$ is a poset. Then, a
\emph{sheaf $\mathscr{S}$ of sets on $P$ with the Alexandroff topology
}consists of the following:

1. For each $p\in P$, a set $\mathscr{S}(p)$, is called the \emph{stalk}
at $p$,

2. For each pair $p\leq q\in P$, there exists a function $\mathscr{S}(p\leq q):\mathscr{S}(p)\rightarrow\mathscr{S}(q)$,
which is called a \emph{restriction function} (\emph{restriction}),
such that

3. For each triple $p\leq q\leq r\in P,$ $\mathscr{S}(p\leq r)=\mathscr{S}(q\leq r)\circ\mathscr{S}(p\leq q)$,
\emph{i.e.} the diagram commutes.
\end{defn}
If any of the conditions above are not satisfied, we do not have a
sheaf. These constructions are called \emph{diagrams} rather than
sheaves. 

Robinson observes that the stalks can have structure: they can be
vector spaces or topological spaces, for example. Suppose that the
given stalks have structure, and the restriction or extension functions
preserve that structure. We are then left with a sheaf \emph{of }that
type of structure. For example, a sheaf of vector spaces that has
linear functions for each restriction map. 

The encoding of systems of equations as sheaves illustrates a number
of consistencies and inconsistencies between the component models
\cite{Robinson16}. Elements of the stalks that are mutually consistent
across the entire sheaf diagram - and thereby the entire system -
are called \emph{sections}. In other words, the output of the combined
system that corresponds to satisfying the system of equations are
the sections. 
\begin{defn}
\cite{Robinson16} A\emph{ global section} of a sheaf $\mathscr{S}$
on a poset $P$ is an element $s$ of the direct product $\prod_{x\in P}\mathscr{S}(x)$
such that for all $x\leq y\in P$, $\mathscr{S}(x\leq y)(s(x))=s(y)$.
A \emph{local section} is defined similarly, however, it is only defined
on a some subset $Q\subseteq P$.
\end{defn}

\subsection{Systems of Equations}

To begin, consider a multi-model system described by a system of equations.
Here, the set of variables $V$ values lie in the set $W_{v}$ for
all $v\in V$. Each variable is further interrelated though a set
of equations $E$ such that each equation $e\in E$ specifies a list
of variables $V_{e}\subset V$ and a subset of solutions $S_{e}\subseteq\prod_{v\in V_{e}}W_{v}$.
Notice that there exist a natural projection functions for each $x\in V_{e}$
\emph{i.e.} $pr_{x}:\prod_{v\in V_{e}}W_{v}\rightarrow W_{x}$. This
allows us to define a poset structure, since these projection functions
restrict to functions on $S_{e}$ such as $pr_{x}:S_{e}\rightarrow W_{x.}$
The poset structure is defined in the following way. Let $P=V\sqcup E$,
\emph{i.e. }variables of $P$ are either variables are equations and
define $e\leq v$ if $v\in V_{e}$. This then defines a partial order
on $P$ if we assume that $\leq$ is reflexive. 
\begin{defn}
\cite{Robinson16} A \emph{sheaf} $\mathscr{E}'$ on $(P,\leq)$ can
be defined by specifying the following:

1. $\mathscr{E}'(v)=W_{v}$ for every variable $v$,

2. $\mathscr{E}'(e)=\prod_{v\in V_{e}}W_{v}$ for every equation $e$,

3. $\mathscr{E}'(e\leq v)=p_{v}$ whenever $e\leq v$. 
\end{defn}
A number of important claims follow about the sections of the system
of equations and how they correspond to the sections of the sheaf.
\begin{prop}
\cite{Robinson16} Assume that each variable $v$ appears in at least
one equation. Then the set of sections of $\mathscr{E}'$ is in one-to-one
correspondence with $\prod_{v\in V_{e}}W_{v}$.\end{prop}
\begin{proof}
Notice that each section of $\mathscr{E}'$ specifies all values of
all variables. This follows since each $v\in P$ and its stalk corresponds
to its respective space of values. Further, specifying the value
of each variable clearly specifies a section of $\mathscr{E}'$. 
\end{proof}
The aggregation of sheaf $\mathscr{E}'$ does not account for the
actual equations; rather it simply specifies the variables which are
involved. To account for this loss of information, we construct a
sub-sheaf $\mathscr{E}$ of $\mathscr{E}'$. 
\begin{defn}
\cite{Robinson16} The \emph{solution sheaf} $\mathscr{E}$ of a system
of equations is given by the following:

1. $\mathscr{E}(v)=W_{v}$ for every variable $v$,

2. $\mathscr{E}(e)=S_{e}$ for every equation $e$,

3. $\mathscr{E}(e\leq v)=p_{v}$ whenever $e\leq v$. 
\end{defn}
Recall that $S_{e}\subseteq\prod_{v\in V_{e}}W_{v}$ is the set of
solutions to $e.$ Further, the next proposition outlines a rather
important fact about the relationship between the sections of $\mathscr{E}$
and the solutions to the system of equations.
\begin{prop}
\cite{Robinson16} The sections of $\mathscr{E}$ consist of solutions
to the simultaneous system of equations.\end{prop}
\begin{proof}
First observe that a section $s$ of $\mathscr{E}$ specifies an element
$s(e)\in S_{e}$ for every equation $e\in E$ that satisfies that
condition. On the other hand, let the solution to a simultaneous system
of equations be given. Then this is a specification of some element
$x\in\prod_{v\in V}W_{v}$ where the projection of x onto $\prod_{v\in V}W_{v}$
lies in $S_{e}$. In other words, an assignment onto each variable
$v\in V$ is given by the following: 
\[
s(v)=pr_{v}x\mbox{, and }s(e)=pr_{\mathscr{E}'(e)}x.
\]
Observe that by construction $s(e)\in S_{e}=\mathscr{S}(e)$.
\end{proof}
Up to this point we have been dealing with systems of \emph{arbitrary
}equations. However, there is generally more structure available to
us. When available, the stalks of the sheaf $\mathscr{E}$ over the
variables can be reduced in size. This results in computational savings.
Further, equations usually take on the form
\[
v_{n+1}=f_{e}(v_{1},\dots,v_{n}).
\]
In the following cases, it is most helpful to employ a \emph{dependency
graph} to visualize the relationships among the variables. 
\begin{defn}
\cite{Robinson16} A system of equations $E$ on variables $V$ is
called \emph{explicit} if there is an injective function $\gamma:E\rightarrow V$
that selects a specific variable from each equation such that each
equation $e\in E$ has the form 
\[
\gamma(e)=f_{e}(v_{1}\dots v_{n})
\]
such that $\gamma(e)\in V(e)$ and $\gamma\notin\{v_{1}\dots v_{n}\}.$
If any variable is outside the image of $\gamma$, then that variable
is said to be \emph{free }or \emph{independent. }Conversely, if a
variable is inside the image of $\gamma$ then that variable is said
to be \emph{dependent. }
\end{defn}
When dealing with an explicit system a dependency graph can be defined
in the following way
\begin{defn}
\cite{Robinson16} A \emph{variable dependency graph }for an explicit
system is a directed graph $G$ whose vertices are are given by the
union $E\bigcup(V\setminus\gamma(E))$. This consists of the set of
equations and free variables such that the following hold: 

1. Free variables have in-degree zero,

2. If $e$ is a vertex of $G$ corresponding to an equation whose
incoming edges are given by $(e_{1}\rightarrow e),\dots,(e_{n}\rightarrow e)$,
then the equation $e\in E$ is of the form 
\[
\gamma(e)=f_{e}(\gamma(e_{1}),\dots\gamma(e_{n})).
\]
\end{defn}
\begin{example}
Recall the Goodwin model defined by (4.11). Then note that the Goodwin
model is an explicit system and its dependency graph is shown
in Figures 11 and 12 in subsection 5.2.
\end{example}
If $E$ is an explicit system of equations with variable V, then it
is possible to construct the explicit solution sheaf $\mathscr{G}$.
The underlying poset for $\mathscr{G}$ is still given by the union
of the variables and the equations, however the stalks and restriction
maps are slightly different. 
\begin{defn}
\cite{Robinson16} The \emph{explicit solution sheaf }$\mathscr{G}$
whose sections are the simultaneous solutions of $E$ is defined by
the following:

1. $\mathscr{G}(v)=W_{v}$ for each variable $v\in V$,

2. $\mathscr{G}(v)=\prod_{x\in V_{e}\setminus\gamma(e)}W_{x}$,

3. $\mathscr{G}(e\leq\gamma(e))=f_{e}$, and 

4. $\mathscr{G}(e\leq v):\prod_{x\in V_{e}\setminus\gamma(e)}W_{x}\rightarrow W_{v}$
is given by an appropriate projection if $v\neq\gamma(e)$.\end{defn}
\begin{example}
The explicit solution sheaf for the Goodwin system defined in (4.11)
is shown in Figure 13 in subsection 5.2. \end{example}
\begin{prop}
\cite{Robinson16} The sections of an explicit solution sheaf $\mathscr{G}$
are in one-to-one correspondence with the simultaneous solutions of
its system of equations. \end{prop}
\begin{proof}
Assume that $e$ is an equation in the explicit system of the form
\[
\gamma(e)=f_{e}(\gamma(e_{1}),\dots\gamma(e_{n})).
\]
Then notice that $S_{e}=\{v_{1},\dots v_{n},f_{e}(\gamma(e_{1}),\dots\gamma(e_{n})):v_{i}\in W_{\gamma(e_{i})}\}.$
Thus, the Proposition follows directly from Proposition 5.2.
\end{proof}

\subsubsection{Ordinary Differential Equations}

The framework developed in the previous section extends nicely to
ordinary differential equations, which is the main object being considered
throughout this paper. Differential equations give rise to sheaves
of solutions, upon which various types of analyses can be conducted.
Consider an ordinary differential equation, given by
\begin{equation}
u'=f(u),
\end{equation}
 where $u\in C^{1}(\mathbb{R},\mathbb{R}^{d})$ is a continuously
differentiable function. There are two ways to consider $u$ and $u'$:
either as one variable or as two separate variables. If we first consider
them as one variable, we note that this is essentially writing (5.1)
as:
\[
0=F(u)=f(u)-\frac{d}{dt}u.
\]
The solutions of (5.1) are the sections of the sheaf diagram:

\[
\xymatrix{C^{1}(\mathbb{R},\mathbb{R}^{d})\\
\{u:F(u)=0\}\subseteq C^{1}(\mathbb{R},\mathbb{R}^{d})\ar[u]^{\mbox{id}}
}
\]
When looking to conduct our analysis, this sheaf lacks some important
information; too much of the structure of (5.1) is hidden in the function
$F$. 

We now turn to considering $u$ and $u'$ as separate variables. The
initial construction of the sheaf yields a similar structure, 
\[
\xymatrix{C^{0}(\mathbb{R},\mathbb{R}^{d}) & C^{1}(\mathbb{R},\mathbb{R}^{d})\\
C^{1}(\mathbb{R},\mathbb{R}^{d})\ar[u]^{f}\ar[ur]^{\mbox{id}}
}
\]
where $u$ is on the top right and $u'$ is on the top left. However,
there is still some information buried. Notice that the solutions
of (5.1) correspond to sections of this sheaf. However, the converse
is not true! Our diagram is missing another equation that links $u$
and $u'$; they are related through differentiation. Including this
relationship yields the following sheaf diagram

\[
\xymatrix{C^{0}(\mathbb{R},\mathbb{R}^{d}) & C^{1}(\mathbb{R},\mathbb{R}^{d})\\
C^{1}(\mathbb{R},\mathbb{R}^{d})\ar[u]^{f}\ar[ur]^{\mbox{id}} & C^{1}(\mathbb{R},\mathbb{R}^{d})\ar[u]_{\mbox{id}}\ar[ul]^{\mbox{\ensuremath{\frac{d}{dt}}}}
}
\]
This sheaf's sections are now in one-to-one correspondence with the
differential equation. 

From here, several stalks of the sheaf can be collapsed together,
without disrupting the space of global sections. A cleaner sheaf arises,
yet it still reflects all of the relevant information: 
\[
\xymatrix{C^{0}(\mathbb{R},\mathbb{R}^{d}) & C^{1}(\mathbb{R},\mathbb{R}^{d})\ar[l]_{\frac{d}{dt}}\\
S\subseteq C^{1}(\mathbb{R},\mathbb{R}^{d})\ar[u]^{f}\ar[ur]^{\mbox{id}}
}
\]
where $S$ is the space of solutions.

\section{Goodwin Models and International Trade}

\subsection{Two Countries with Horizontal Trade}

It is at this point that we return to the Goodwin growth model to
apply some of the sheaf theoretic applications just introduced. Recall
the basic framework of the Goodwin model. In the original model, and
all the updates that were discussed, one assumption reminded constant:
the economy that was being analyzed was a closed economy. One possible
way to extend the Goodwin model, thereby making it more representative
of phenomena, is to consider an open economy that is trading with
other countries. The framework employed here mimics that of Ishiyama
\cite{Dummy1}, though with some divergences. For example, Ishiyama
builds his two country Goodwin model through difference equations,
while this paper will continue to do so with differential equations.
These differences are subtle, and while they don't affect phenomena
such as equilibrium, dynamic results can vary. 

The motivation for expanding the Goodwin model is three-fold. First,
aside for the work done by Ishiyama, it seems this idea of extending
the Goodwin framework to an international setting has not been done.
It is not entirely clear why this is the case. The Goodwin framework
elegantly models fluctuations and cycles within an economy, thus making
it an attractive candidate to expand. It appears that this question
was not tackled in the past not because it is uninteresting or of
little value, but may be because the mathematics that surround extending
the model to many countries is daunting. Second, extending the Goodwin
model requires a systematic approach. As the number of countries increases,
it becomes ever more useful to have a consistent framework which allows
us to conduct our analysis. Sheaf theory provides that framework.
Third, and most relevant toward the fields of economics and policy,
the types of questions that can be asked and answered using a sheaf
theoretic approach allows us to provide new insight into topics that
were once difficult to answer. 

For simplicity, we begin by considering a two-country two-good model
in which we assume horizontal trade occurs, \emph{i.e.} the goods
traded have no relation to each other in terms of factors of production.
Considering two goods is suitable from a theoretical standpoint. However,
if we want to test how good this model is in reality, we would consider
a basket of goods, rather than a single good. Consider, for example,
guns and butter. We formally employ the following assumptions: \\
\textbf{}\\
\textbf{(A1)} There are two countries that produce different goods.\textbf{}\\
\textbf{(A2)} Each country has a steady growth of technical progress.\\
\textbf{(A3)} Each country has a steady growth of the labor force.\\
\textbf{(A4)} In each county there exist only two homogeneous and
non-specific factors of production: capital and labor.\textbf{}\\
\textbf{(A5)} There is no mobility of either capital or labor between
the two countries.\textbf{}\\
\textbf{(A6)} All quantities are real and net.\\
\textbf{(A7)} All wages are consumed; all profits are saved and invested.\\
\textbf{(A8)} The utility of the consumer is maximized at a positive
combination of two goods. Thus horizontal trade continues permanently.\\
\textbf{(A9)}The change in the money wage rate is determined by a
Phillips Curve. \\
\textbf{(A10)} The capital-output ratio is constant.\\
\textbf{(A11)} The two countries use the same currency.\\
\\
From these assumptions, and the derivation of the Goodwin model above,
we can represent our two countries as four unique differential equations:
\begin{align}
\dot{v}_{i}= & v_{i}(t)\left(\frac{1}{\sigma_{1}}-(\alpha_{i}+\beta_{i})-\frac{u_{i}(t)}{\sigma_{i}}\right)\label{eq:6.1}\\
\dot{u}_{i}= & \frac{u_{i}(t)}{p_{i}(t)}\left(-(\alpha_{i}+\gamma_{i})+(\rho_{i}v_{i}(t))\right)\nonumber 
\end{align}
for $i=1,2$. Note that the equations for each country are in the
exactly the same form, and defined by exactly the same variables and
constants, but they differ in value depending on the country in question. 

Concerning assumption \textbf{(A8)}, a utility function is defined,
which shows the preferences of the representative workers in each
country as indifference curves. Figure 16 graphically represents this for country
one, though country two is exactly the same. Consumers in each country
can buy two different goods $x_{1}$ and $x_{2}$. As stated by \textbf{(A7)},
all wages earned by workers are then spent on some bundle of these
two goods. Mathematically this is described in the equation 
\[
I=p_{1}x_{1}+p_{2}x_{2}.
\]
For a bundle of goods to be purchased, it must lie on or inside of
the budget constraint. Further, the indifference curves represent
some combination of $x_{1}$ and $x_{2}$ such that the utility achieved
from all of those different bundles is the same; \emph{i.e.} they
are indifferent about which bundle they choose. For consumers in country
one, this is represented mathematically as 
\[
U_{1}=x_{1}^{\theta_{1}+\rho}\cdot x_{2}^{(1-\theta_{1})^{-2}}
\]
where $\theta_{1}$ is a parameter strictly greater than zero and
strictly less than one. The symbol $\rho$ is defined as the price
proportion, \emph{i.e. }$\rho=\frac{p_{2}}{p_{1}}$. This utility
function corresponds to a special Cobb-Douglas function which reflects
changes in the price proportion. In other words, built into the utility
function is a \emph{bias} with respect to the goods produced in different
countries. 

\begin{figure}
\begin{centering}
\includegraphics[scale=0.7]{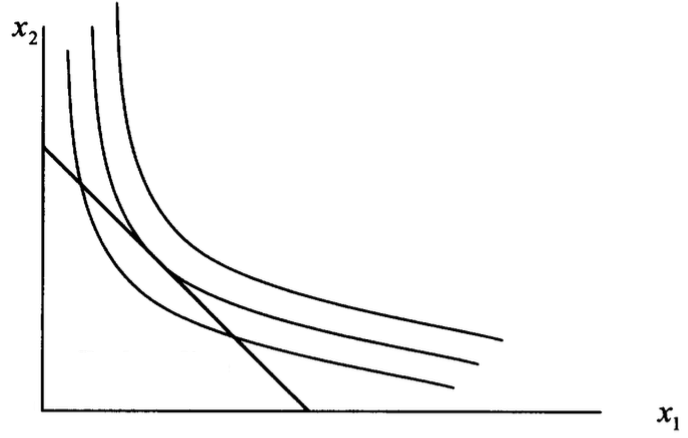}\caption{Indifference Curves and Budget Line \cite{Dummy1}}

\par\end{centering}

\end{figure}

Utility is maximized when the budget line and indifference curve are
tangent to each other. Given the utility function, Ishiyama showed
that the optimal combination for country one on a budget line for
a given price vector is:

\[
(x_{1},x_{2})=\left(\frac{(\theta_{1}p_{1}+p_{2})a_{1}u_{1}(t)v_{1}(t)N_{1}}{p_{1}(p_{1}+p_{2})},\frac{(1-\theta_{1})p_{1}a_{1}u_{1}(t)v_{1}(t)N_{1}}{p_{2}(p_{1}+p_{2})}\right).
\]
This is the demand function for a given price vector. Note that country
one consumes domestic goods by $\theta_{1}$ of its income.

Similarly, country two's utility curve can be defined as follows:
\[
U_{2}=x_{1}^{(1-\theta_{2}){}^{2}}\cdot x_{2}^{\theta_{2}+\rho^{-2}}
\]
 and the optimal combination of goods for a given price vector is:
\[
(x_{1},x_{2})=\left(\frac{(1-\theta_{2})p_{2}a_{2}u_{2}(t)v_{2}(t)N_{2}}{p_{1}(p_{1}+p_{2})},\frac{(\theta_{1}+\theta_{2}p_{2})a_{2}u_{2}(t)v_{2}(t)N_{2}}{p_{2}(p_{1}+p_{2})}\right).
\]
These show the preference and consumption behavior of country two. 

\begin{figure}

\begin{centering}
\includegraphics[scale=0.7]{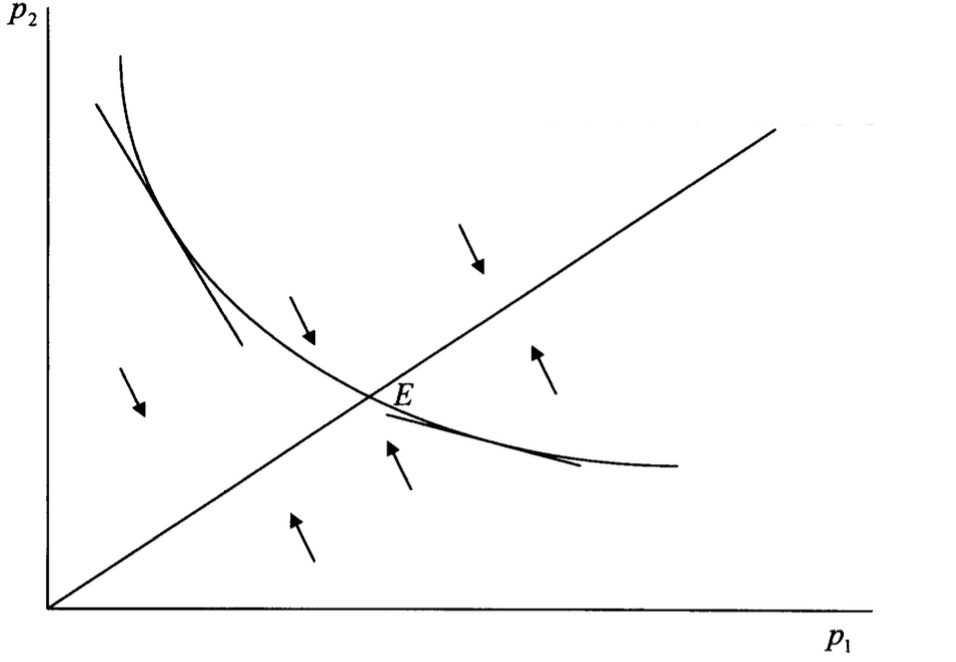}\caption{Equilibrium of Prices \cite{Dummy1}}

\par\end{centering}

\end{figure}

Given these two demand curves, we can write down equations to describe
price adjustments in each country, where the excess demand functions
are on the right-hand side of the equation:
\begin{align}
\dot{p}_{1}= & (1-\theta_{2})p_{2}a_{2}u_{2}(t)v_{2}(t)N_{2}-(1-\theta_{1})p_{1}a_{1}u_{1}(t)v_{1}(t)N_{1}\label{eq:6.2}\\
\dot{p}_{2}= & (1-\theta_{1})p_{1}a_{1}u_{1}(t)v_{1}(t)N_{1}-(1-\theta_{2})p_{2}a_{2}u_{2}(t)v_{2}(t)N_{2}.\nonumber 
\end{align}

\begin{prop}
The Short Run Equilibrium (i.e. balance of trade) price proportion
$\rho^{*}=\frac{(1-\theta_{1})a_{1}u_{1}(t)v_{1}(t)N_{1}}{(1-\theta_{2})a_{2}u_{2}(t)v_{2}(t)N_{2}}$
is stable. \end{prop}
\begin{proof}
To show the stability of the price proportion, consider system (6.2)
rewritten in the form $\boldsymbol{\dot{p}}=A\cdot\boldsymbol{p}$:
\[
\left[\begin{array}{c}
\dot{p_{1}}\\
\dot{p}_{2}
\end{array}\right]=\left[\begin{array}{cc}
-(1-\theta_{1})a_{1}u_{1}(t)v_{1}(t)N_{1} & (1-\theta_{2})a_{2}u_{2}(t)v_{2}(t)N_{2}\\
(1-\theta_{1})a_{1}u_{1}(t)v_{1}(t)N_{1} & -(1-\theta_{2})a_{2}u_{2}(t)v_{2}(t)N_{2}
\end{array}\right]\left[\begin{array}{c}
p_{1}\\
p_{2}
\end{array}\right]
\]
and $A$'s eigenvalues, which are found by evaluating:
\begin{align*}
det(\lambda I-A)= & det\left[\begin{array}{cc}
-\lambda-(1-\theta_{1})a_{1}u_{1}(t)v_{1}(t)N_{1} & (1-\theta_{2})a_{2}u_{2}(t)v_{2}(t)N_{2}\\
(1-\theta_{1})a_{1}u_{1}(t)v_{1}(t)N_{1} & -\lambda-(1-\theta_{2})a_{2}u_{2}(t)v_{2}(t)N_{2}
\end{array}\right]\\
= & [-\lambda-(1-\theta_{1})a_{1}u_{1}(t)v_{1}(t)N_{1}]\cdot[-\lambda-(1-\theta_{2})a_{2}u_{2}(t)v_{2}(t)N_{2}]\\
 & -[(1-\theta_{2})a_{2}u_{2}(t)v_{2}(t)N_{2})\cdot((1-\theta_{1})a_{1}u_{1}(t)v_{1}(t)N_{1})]\\
= & \lambda^{2}+\lambda[((1-\theta_{1})a_{1}u_{1}(t)v_{1}(t)N_{1})+((1-\theta_{2})a_{2}u_{2}(t)v_{2}(t)N_{2})]\\
= & \lambda^{2}[\lambda+((1-\theta_{1})a_{1}u_{1}(t)v_{1}(t)N_{1})+((1-\theta_{2})a_{2}u_{2}(t)v_{2}(t)N_{2})]
\end{align*}
 The eigenvalues are $0$ and $-[(1-\theta_{1})a_{1}u_{1}(t)v_{1}(t)N_{1})+((1-\theta_{2})a_{2}u_{2}(t)v_{2}(t)N_{2}]$.
Thus the price proportion is stable. 
\end{proof}
Figure 17 shows graphically that the short run equilibrium is stable.
When the parameter $\theta_{1}$ is equal to $\theta_{2}$, the slope
of a local change in the price vector $\boldsymbol{p}$ is expressed
as a tangent to the hyperbolic curve given by 
\[
p_{2}=\frac{p_{1}(0)p_{2}(0)}{p_{1}}
\]
where $(p_{1}(0),p_{2}(0))$ is some initial price vector. The tangent
lines become smaller as the point comes near the equilibrium ray.
This is the line where prices are no longer changing and is given
by:
\[
p_{2}=\frac{a_{1}u_{1}(t)v_{1}(t)N_{1}}{a_{2}u_{2}(t)v_{2}(t)N_{2}}\cdot p_{1}.
\]
Thus if the adjustment is sufficiently flexible, the trade between
two countries at period $t$ is conducted at an equilibrium, or the
positive intersection point of the hyperbolic curve and equilibrium
ray. $E$ is defined to be the point: 
\[
E=\left(\sqrt{\frac{p_{1}(0)p_{2}(0)a_{2}u_{2}(t)v_{2}(t)N_{2}}{a_{1}u_{1}(t)v_{1}(t)N_{1}}},\sqrt{\frac{p_{1}(0)p_{2}(0)a_{1}u_{1}(t)v_{1}(t)N_{1}}{a_{2}u_{2}(t)v_{2}(t)N_{2}}}\right).
\]
 Thus we summarize our two country horizontal trade model with the
following equations:
\begin{align}
\dot{v}_{i}= & v_{i}(t)\left(\frac{1}{\sigma_{i}}-(\alpha_{i}+\beta_{i})-\frac{u_{i}(t)}{\sigma_{i}}\right)\nonumber \\
\dot{u}_{i}= & \frac{u_{i}(t)}{p_{i}(t)}\left(-(\alpha_{i}+\gamma_{i})+(\rho_{i}v_{i}(t))\right)\nonumber \\
\dot{p}_{1}= & \sqrt{\frac{p_{1}(0)p_{2}(0)a_{2}u_{2}(t)v_{2}(t)N_{2}}{a_{1}u_{1}(t)v_{1}(t)N_{1}}}\label{eq:6.3}\\
\dot{p}_{2}= & \sqrt{\frac{p_{1}(0)p_{2}(0)a_{1}u_{1}(t)v_{1}(t)N_{1}}{a_{2}u_{2}(t)v_{2}(t)N_{2}}}\nonumber 
\end{align}
for $i=1,2$. This system includes six sets of state variables. 

We now explore the existence on long-run equilibria of the system.
Recall the equilibrium of the Goodwin model when considering just
one country. The system has two fixed points: the trivial fixed point
at the origin and the non-trivial fixed point. The non-trivial fixed
points are
\begin{align*}
v^{*}= & \frac{\alpha+\gamma}{\rho}\\
u^{*}= & 1-\sigma(\alpha+\beta).
\end{align*}
The Jacobian evaluated at the non-trivial fixed points is
\[
\boldsymbol{J}=\left(\begin{array}{cc}
0 & -\frac{\alpha+\gamma}{\sigma\rho}\\
\rho(1-\sigma(\alpha+\beta) & 0
\end{array}\right).
\]
Notice that the equations describing a country's dynamics in this
section are slightly modified from the original Goodwin model. The
change is small, but it is present in the equation: 
\[
\dot{u}_{i}=\frac{u_{i}(t)}{p_{i}(t)}\left(-(\alpha_{i}+\gamma_{i})+(\rho_{i}v_{i}(t))\right).
\]
The function is updated by dividing $u_{i}$ by $p_{i}$. However,
this doesn't change the equilibrium point $v^{*}$. For the two country
model, since each county exhibits its own unique Goodwin cycle, the
non-trivial fixed points are
\begin{align*}
v_{i}^{*}= & \frac{\alpha_{i}+\gamma_{i}}{\rho_{i}}\\
u_{i}^{*}= & 1-\sigma_{i}(\alpha_{i}+\beta_{i})
\end{align*}
for $i=1,2$, and it has a similar Jacobian matrix. Thus, each Goodwin
model is a conservative dynamical system where every point is in a
closed orbit around the non-trivial fixed points which are asymptotically
stable. Notice that $p_{1}^{*}$ and $p_{2}^{*}$ occurs when the
following are true: first, prices are in equilibrium \emph{i.e. }we
are somewhere along the equilibrium ray and when both $u_{i}$ and
$v_{i}$ are in equilibrium. Thus, we get 
\begin{align*}
p_{1}^{*}= & \sqrt{\frac{p_{1}(0)p_{2}(0)a_{2}u_{2}^{*}v_{2}^{*}N_{2}}{a_{1}u_{1}^{*}v_{1}^{*}N_{1}}}\\
p_{2}^{*}= & \sqrt{\frac{p_{1}(0)p_{2}(0)a_{1}u_{1}^{*}v_{1}^{*}N_{1}}{a_{2}u_{2}^{*}v_{2}^{*}N_{2}}}.
\end{align*}

From the analysis above, it is shown that the two country Goodwin
model describing horizontal international trade has a meaningful equilibrium
point for real variables. Next, we will consider the mechanism of
dynamics for the system. Recall the behavior of worker's share and
employment rate when considering a single Goodwin model, described
in Figure 10. When labor share of national income ($u$) of the country
$i$ is greater than $u_{i}^{*}$ and the employment ratio ($v$)
is \emph{also }greater than $v_{i}^{*}$, the main economic variables
considered in the Goodwin model all fall down because the pressure
put on capitalist's profits weakens investment activities. The
wage continues to rise until $v_{i}$ becomes less than or equal to
$v_{i}^{*}$. 

\begin{figure}
\centering{}\includegraphics[scale=0.55]{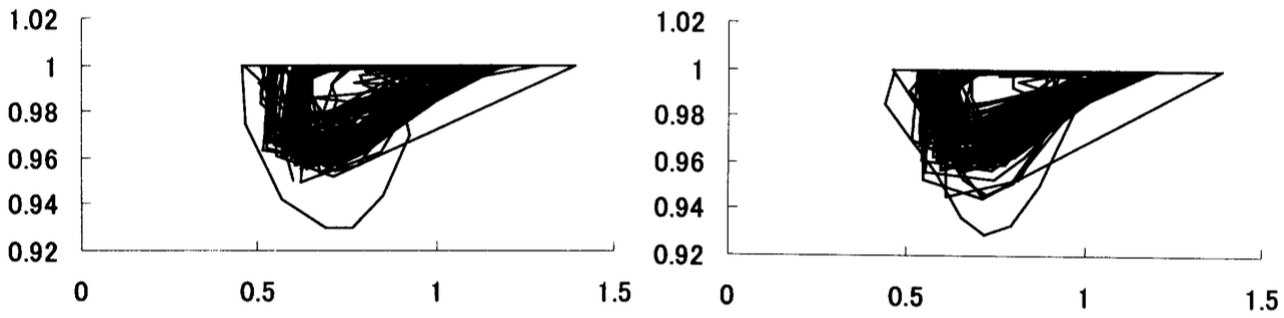}\caption{Country one (left) with trade and Country two (right) with trade \cite{Dummy1}}
\end{figure}

However, when considering horizontal trade, as our two country model
does, the behavior of labor share in country one, $u_{1}$, is unclear.
In the case of our system, $u_{1}$, depends on many more variables
than just $v_{1}$, as it did in the original Goodwin. This
fact will be made more clear in the next section where we sheafify
the two country Goodwin model and these dependencies can be seen graphically.
Further, the nominal capital stock also changes by the influence of
one country's behavior. For example, if foreign workers demand domestic
products such that the rise of its price can be offset by the growth
of money, the domestic labor share falls. Thus we expect the trajectories
generated by our dynamic system to exhibit much more complex behavior
than that of the Goodwin model \emph{without }horizontal trade. 

\begin{figure}
\centering{}\includegraphics[scale=0.5]{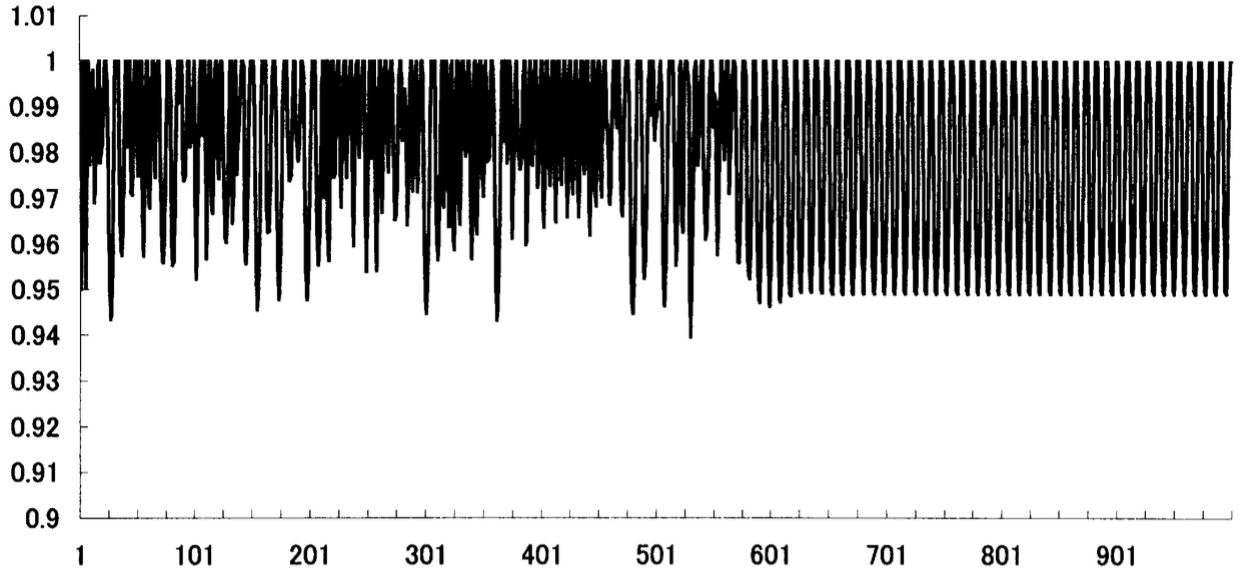}\caption{Country one's Dynamics - $\sigma_{2}=2.6$ \cite{Dummy1}}
\end{figure}

\begin{figure}
\centering{}\includegraphics[scale=0.25]{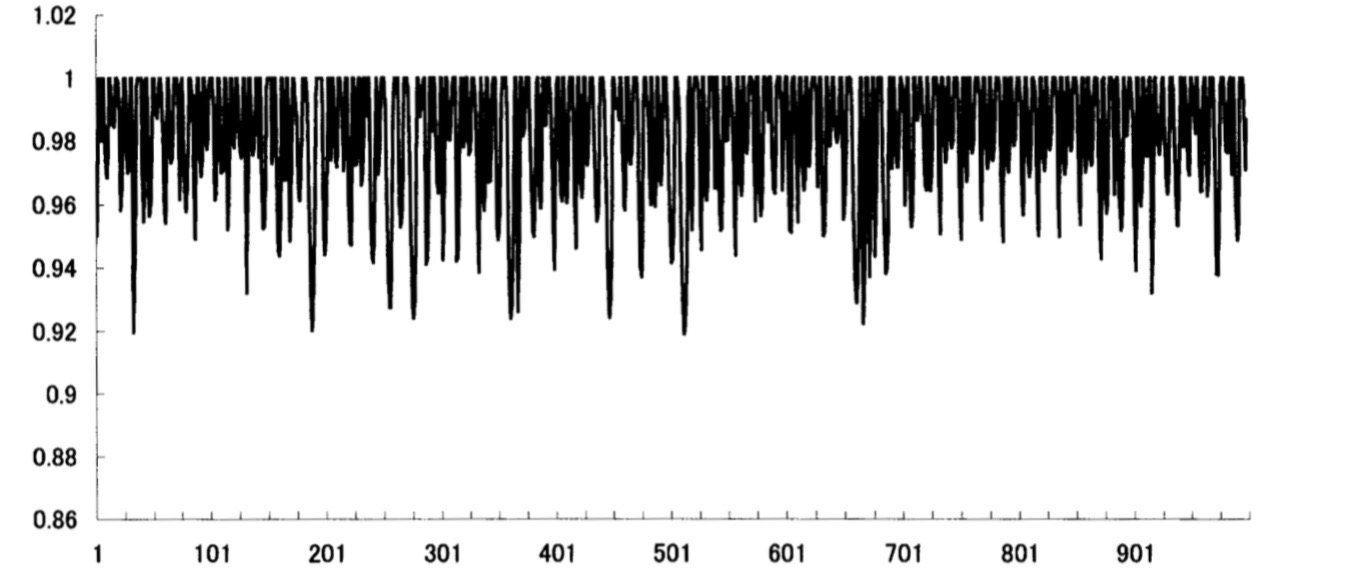}\caption{Country one's Dynamics - $\sigma_{2}=4.7$ \cite{Dummy1}}
\end{figure}

Ishiyama looks to confirm these conjectures through computer simulations.
There are two different cases that are considered. First, the capital-output
ratios between the two countries are equal, \emph{i.e. $\sigma=\sigma_{1}=\sigma_{2}$};
second, the more general case of capital-output ratios not equal to
each other is considered. To examine the relationship between the
lack and presence of horizontal trade, note that when horizontal trade
is not present, we are presented with the original Goodwin model,
which generates closed orbit periodic solutions, which is expected
of a conservative dynamical system. Next, the dynamics between the
two countries is considered when horizontal trade is introduced \emph{and
}capital-output ratios are equal, and is represented in Figure 18.
Country one is on the left side and country two is on the right. Note,
the horizontal axis represents labor's share of national income, while
the vertical axis represents the employment ratio. 

\begin{table}
\centering{}%
\begin{tabular}{|c|c|}
\hline 
System Dynamics & $\sigma_{2}$\tabularnewline
\hline 
\hline 
Chaotic & 2.0\tabularnewline
\hline 
Limit Cycle & 2.2\tabularnewline
\hline 
Chaotic & 2.3-2.5\tabularnewline
\hline 
Limit Cycle & 2.6\tabularnewline
\hline 
Chaotic & 2.7-3.1\tabularnewline
\hline 
Chaotic & 3.2\tabularnewline
\hline 
Limit Cycle & 3.3-3.6\tabularnewline
\hline 
Chaotic & 3.7, 3.8\tabularnewline
\hline 
Limit Cycle & 3.9-4.1\tabularnewline
\hline 
Chaotic & 4.2-4.7\tabularnewline
\hline 
Limit Cycle & 4.8\tabularnewline
\hline 
\end{tabular}\caption{Dynamics of System with varying $\sigma_{2}$'s \cite{Dummy1}}
\end{table}

Next, the more general case is considered, where capital output ratios
are not equal. Different values are assigned to each country, and
the fluctuations within country one are considered. First,
$\sigma_{2}=2.6$ is considered. Figure 19 shows that over time, country
one's economy reaches a limit cycle, in spite of some initial chaotic
behavior. On the other hand, when $\sigma_{2}=4.7$ is considered,
chaotic behavior appears to persist throughout, which is represented
in Figure 20. Moreover, Ishiyama conducted many different simulations
using various values for $\sigma_{2}$ while keeping $\sigma_{1}$
constant. The properties of fluctuations are classified in Table 1.

\subsection{Sheafifying The Two Country Goodwin Model}

The sheaf theoretic concepts that were introduced in the last section
are interesting from a purely mathematical standpoint. But, their
applications to complex modeling make them a great asset to build
and expand models in a systematic way. They also present a graphical
understanding of how variables relate to each other with varying degrees
of dependency graphs. Further, the notion of being able to ``stitch
together'' a number of simple sub-models to form a more complex multi-model
via sheaf theory has far reaching applications, such as the extended
Goodwin model just introduced. 

Recall the sheaf diagram made for the Goodwin model in Figure 13.
Using that construction process as a model, we wish to build out a
sheaf diagram for the extended Goodwin model using the same process.
Our expectation, given the construction of the two country system,
is that our sheaf diagram will in some way be made up of \emph{two
}original Goodwin models and they will be interlinked in some way.
The two separate original Goodwin models each represent their respective
country. Further, this graphical construction lend some intuition
as to how to extend this to beyond two countries.

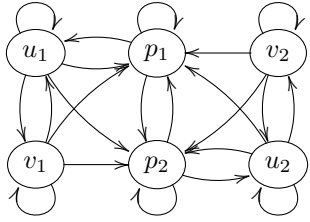
\begin{figure}
\[
\xymatrix{*++[o][F]{u_{1}}\ar@/_{.5pc}/[d]\ar@(lu,ru)\ar@/_{.5pc}/[dr]\ar@/_{.5pc}/[r] & *++[o][F]{p_{1}}\ar@/_{.5pc}/[d]\ar@(lu,ru)\ar@/_{.5pc}/[l] & *++[o][F]{v_{2}}\ar@/_{.5pc}/[d]\ar@(lu,ru)\ar@/^{.5pc}/[dl]\ar[l]\\*
++[o][F]{v_{1}}\ar@/_{.5pc}/[u]\ar@(rd,ld)\ar@/^{.5pc}/[ur]\ar[r] & *++[o][F]{p_{2}}\ar@/_{.5pc}/[u]\ar@(rd,ld)\ar@/_{.5pc}/[r] & *++[o][F]{u_{2}}\ar@/_{.5pc}/[u]\ar@(rd,ld)\ar@/_{.5pc}/[ul]\ar@/_{.5pc}/[l]
}
\]
\caption{Dependency Graph of 2 Countries}

\end{figure}

\begin{figure}

\centering{}\includegraphics[scale=0.5]{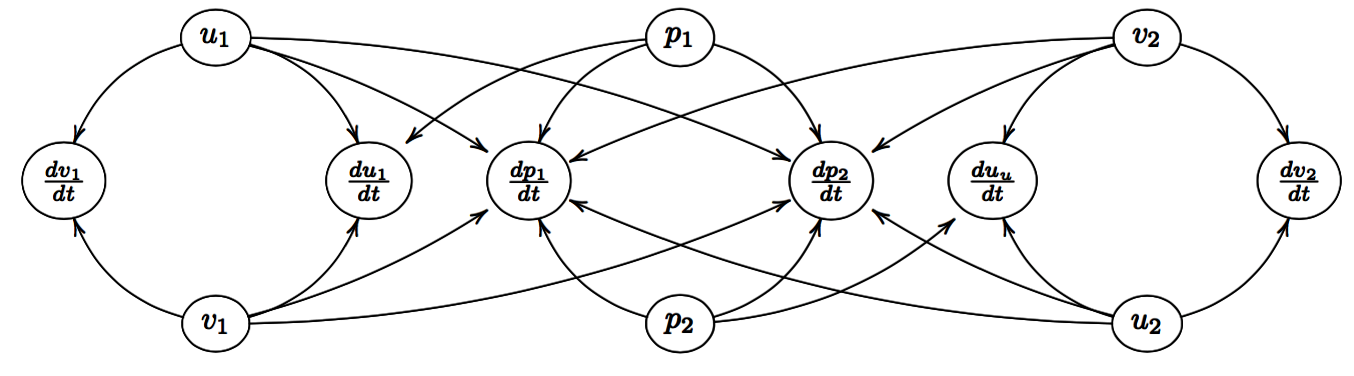}\caption{Expanded Dependency Graph of 2 Countries}
\end{figure}

Given our system of equations (6.3) for our extended Goodwin model,
we can systematically construct the related dependency diagrams and
sheaves. The easiest way to gain an understanding for how the solutions
of (6.3) behave is to build out the first dependency diagram, or the
visual representation of the causal relationships between the variables,
seen in Figure 21. Recall that an arrow from one variable to the next
- \emph{i.e. $u_{1}\rightarrow v_{1}$ }- states that the future value
of $v_{1}$ is partially dependent on the current value of $u_{1}$.
From this first dependency diagram, we begin to see the general structure
of the system. It is easy to isolate the two individual countries
and the variables that interlink them - \emph{i.e. }$p_{1}$ and $p_{2}$.
Taking this dependency diagram one step further to encode more information
- the fact that derivatives of these functions are state variables
- we end up with Figure 22. While more information has been encoded,
it still fails to fully describe the relationship between the derivative
and its state variable, since, for example, both $u_{1}$ and $v_{1}$
determine the value of $\frac{dv_{1}}{dt}$. This is fully remedied
by reinterpreting the arrows as functional relationships between variables
and considering the pair of variables to determine the functional
dependence between variables and their derivatives. After this reinterpretation,
we are left with Figures 23 and 27. 

\begin{figure}
\centering{}\includegraphics[scale=0.4]{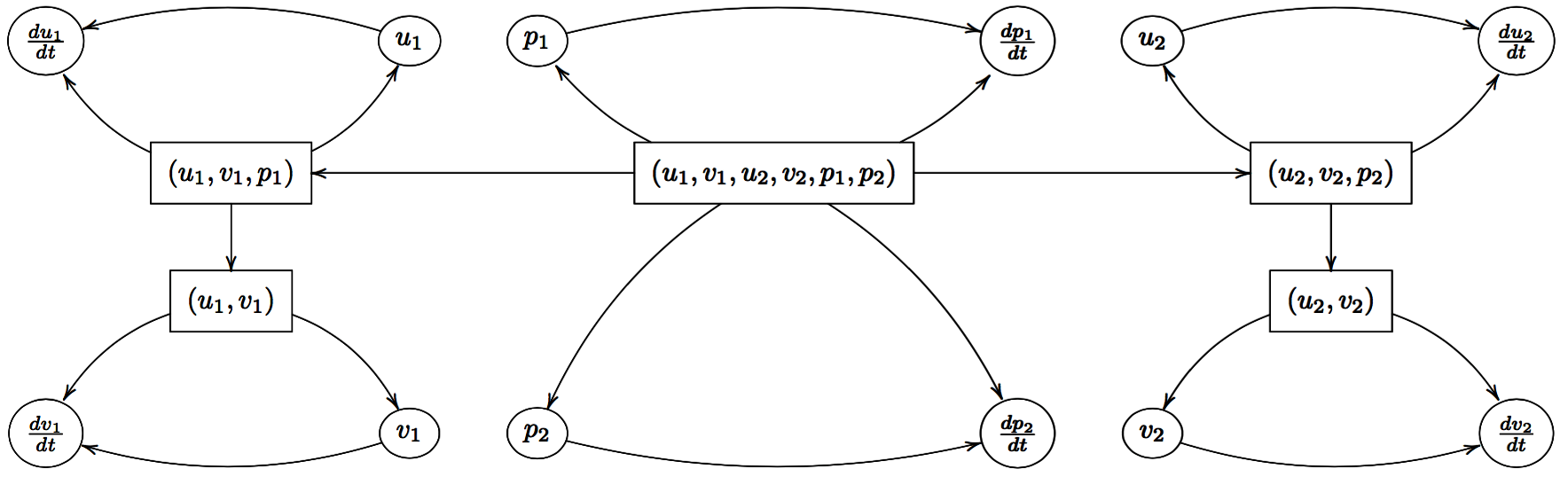}\caption{Functional Dependency Graph of 2 Countries - Variable Names}
\end{figure}

\begin{figure}
\centering{}\includegraphics[scale=0.4]{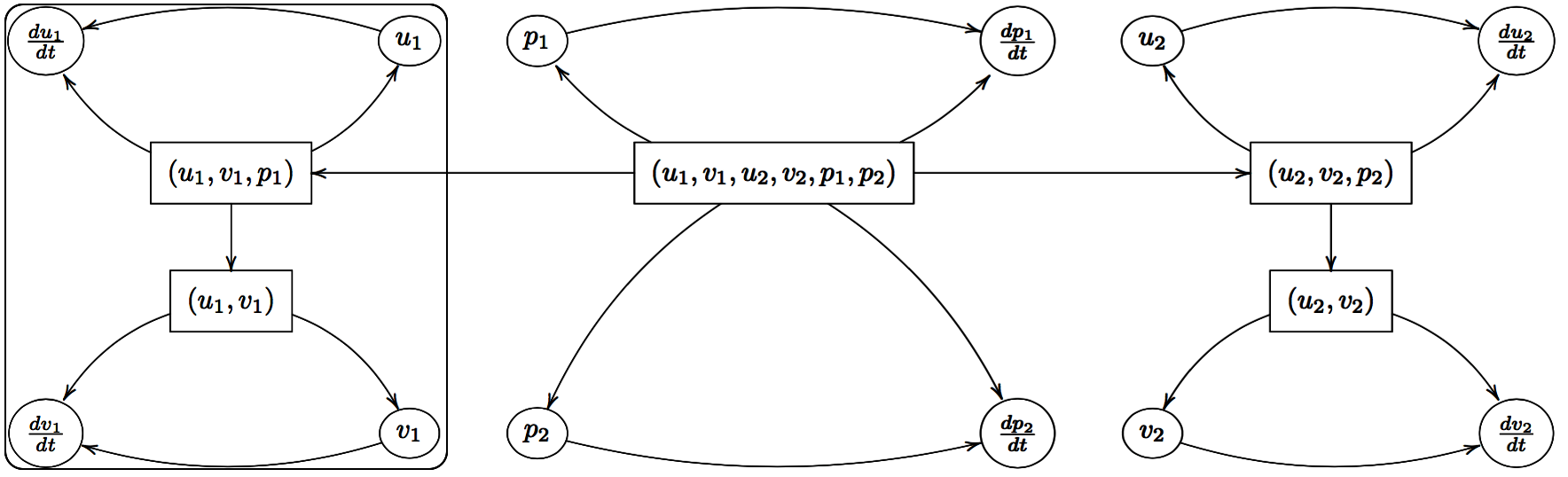}\caption{Country one Sub-Diagram Embedded Within 2 Country Sheaf Diagram}
\end{figure}
Through this process, we have successfully sheafified the Goodwin
Model of two countries engaged in horizontal trade. There are a few
things to note about the diagram. First, each country does not exactly
represent the original Goodwin model, since the new country specific
model introduced by Ishiyama contains a price variable. That being
said, as expected, the diagram can be broken down into three separate
sub-diagrams. On the left, we notice the sub-diagram representing
country one's Goodwin model. Similarly on the right, we notice the
sub-diagram corresponding to county two's Goodwin model. The sub-diagram
in the middle corresponds to the interaction between the two countries
via horizontal trade. These are illustrated in Figures 24, 25, and
26; The boxes outline the sub-sheaf diagrams present within the overall
sheaf. This lends a more intuitive understanding of how the two country
system is constructed; each of these sub-diagrams are ``stitched''
together yielding the larger diagram.

\begin{figure}
\begin{centering}
\includegraphics[scale=0.4]{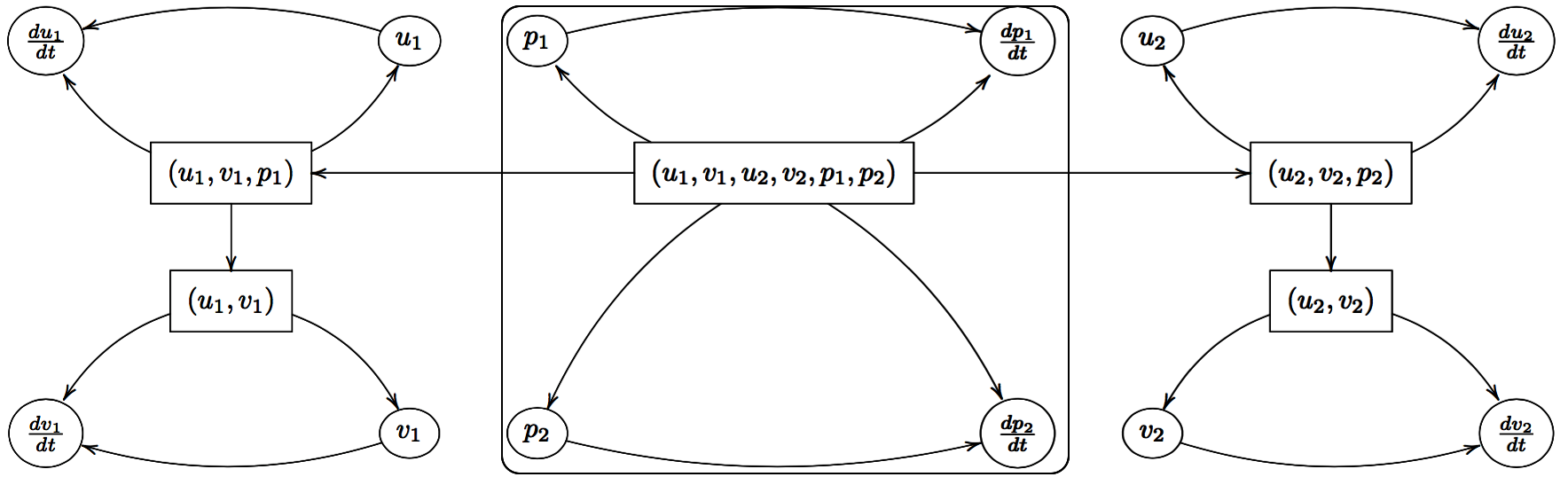}\caption{Price Sub-Diagram Embedded Within 2 Country Sheaf Diagram}

\par\end{centering}

\end{figure}

\begin{figure}
\centering{}\includegraphics[scale=0.4]{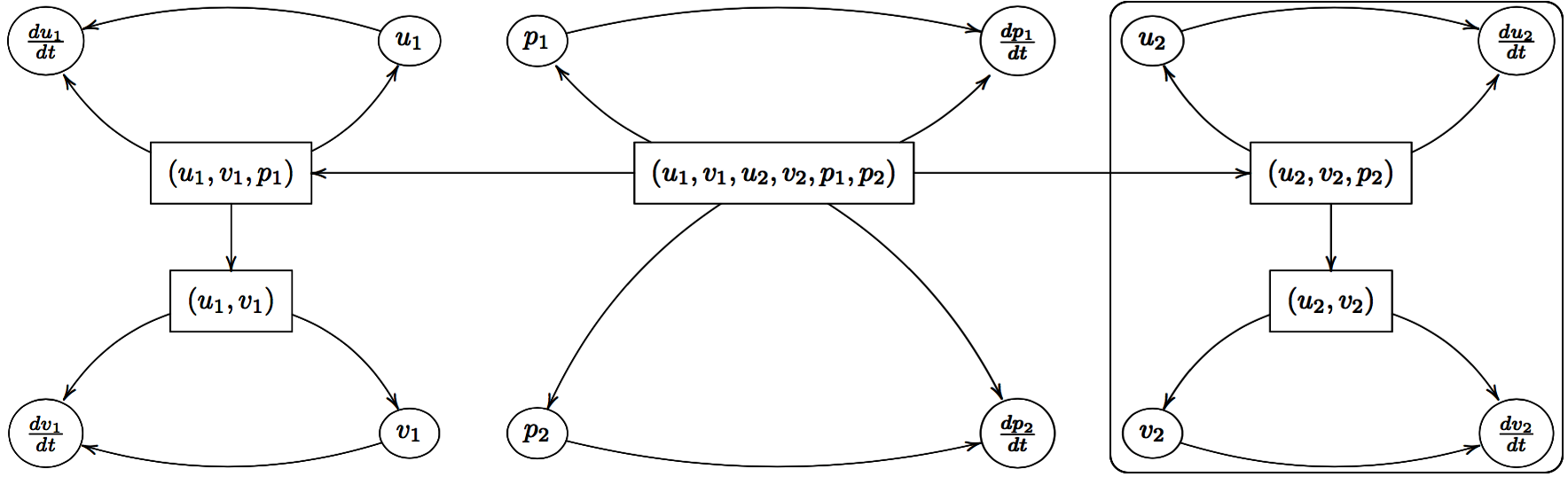}\caption{Country two Sub-Diagram Embedded Within 2 Country Sheaf Diagram}
 
\end{figure}
It is clear from its construction that the sheaf is quite symmetric.
Actually, it is important to point out that Figure 23 and 27 aren't
quite sheaves yet since the diagram most likely does not commute -
one of the requirement for a diagram to be a sheaf. Nevertheless,
it is simple enough to consider the the largest sub-diagram that is
a sheaf. This is done by considering the solution spaces on the state
spaces, \emph{i.e. $S_{1}\subset C^{1}(\mathbb{R,R}^{3})$, $S_{2}\subset C^{1}(\mathbb{R,R}^{6})$}
and $S_{3}\subset C^{1}(\mathbb{R,R}^{3})$. Here, $S_{1}$ and $S_{3}$
represent the set of solutions satisfying the system of equations
describing country one and two respectively. $S_{2}$ is the set
of solutions satisfying the set of functions which describes the price-trading
scheme within our model. We can formally define $S_{1}$, $S_{2}$,
$S_{3}$ in the following way:
\begin{itemize}
\item $S_{1}=\left\{ (u_{1},v_{1},p_{1}:\begin{array}{c}
\dot{v}_{1}=v_{1}(t)\left(\frac{1}{\sigma_{1}}-(\alpha_{1}+\beta_{1})-\frac{u_{1}(t)}{\sigma_{1}}\right)\\
\dot{u}_{1}=\frac{u_{1}(t)}{p_{1}(t)}\left(-(\alpha_{1}+\gamma_{1})+(\rho_{1}v_{1}(t))\right)
\end{array}\right\} $
\item $S_{2}=\left\{ (u_{1},v_{1},p_{1},u_{2},v_{2},p_{2}):\begin{array}{c}
\dot{p}_{1}=\sqrt{\frac{p_{1}(0)p_{2}(0)a_{2}u_{2}(t)v_{2}(t)N_{2}}{a_{1}u_{1}(t)v_{1}(t)N_{1}}}\\
\dot{p}_{2}=\sqrt{\frac{p_{1}(0)p_{2}(0)a_{1}u_{1}(t)v_{1}(t)N_{1}}{a_{2}u_{2}(t)v_{2}(t)N_{2}}}
\end{array}\right\} $
\item $S_{3}=\left\{ (u_{2},v_{2},p_{2}):\begin{array}{c}
\dot{v}_{2}=v_{2}(t)\left(\frac{1}{\sigma_{2}}-(\alpha_{2}+\beta_{2})-\frac{u_{2}(t)}{\sigma_{2}}\right)\\
\dot{u}_{2}=\frac{u_{2}(t)}{p_{2}(t)}\left(-(\alpha_{2}+\gamma_{2})+(\rho_{2}v_{2}(t))\right)
\end{array}\right\} $
\end{itemize}
With this consideration, the diagram now is path independent and thus
we are left with a sheaf. 

\begin{figure}
\centering{}\includegraphics[angle=270,scale=0.5]{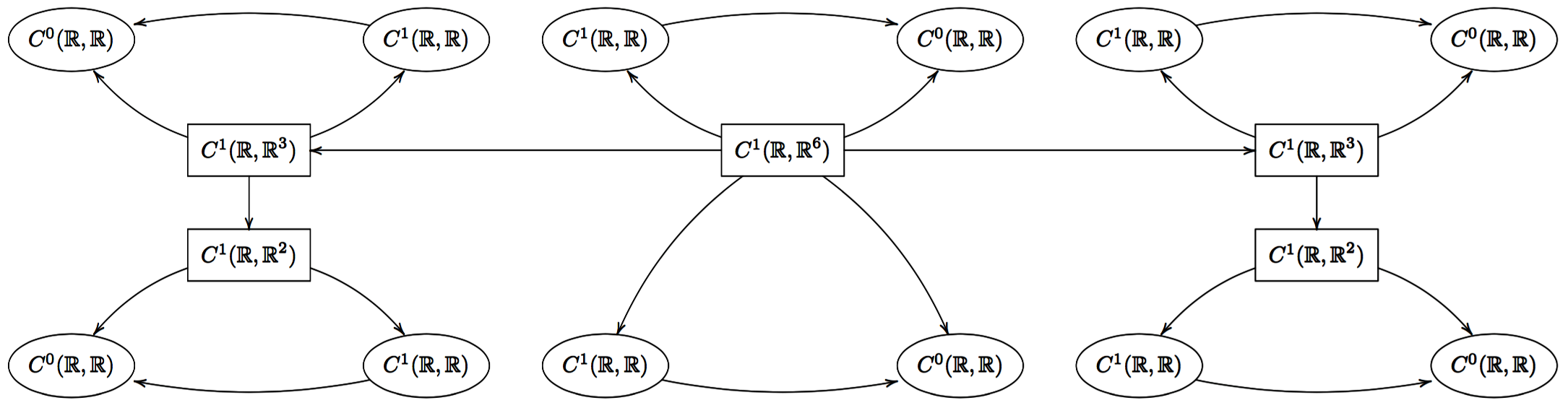}\caption{Functional Dependency Graph of 2 Countries - Variable Spaces}
\end{figure}

\subsection{Sheaf theoretic Analysis}

The dependency graphs presented above are a good introduction to the
power that sheaves have when studying systems of differential equations.
These dependency graphs lend an intuition to how the system is structured
and how variables interact with each other. Sheaf theory goes beyond
simply creating these dependency graphs, though. We turn to asking
questions about properties of the sheaf, such as global and local
sections, and how far can local sections be extended throughout the
system.

One of the global sections for this sheaf is quite easy to compute.
Recall Proposition 5.2 which states the sections of a sheaf are in
one-to-one correspondence with the simultaneous solution of its system
of equations. Thus, then, one global section of our sheaf is simply
the equilibrium of our system
\begin{align*}
v_{i}^{*}= & \frac{\alpha_{i}+\gamma_{i}}{\rho_{i}}\\
u_{i}^{*}= & \frac{1}{\sigma_{i}}-(\alpha_{i}+\beta_{i})\\
p_{1}^{*}= & \sqrt{\frac{p_{1}(0)p_{2}(0)a_{2}u_{2}^{*}v_{2}^{*}N_{2}}{a_{1}u_{1}^{*}v_{1}^{*}N_{1}}}\\
p_{2}^{*}= & \sqrt{\frac{p_{1}(0)p_{2}(0)a_{1}u_{1}^{*}v_{1}^{*}N_{1}}{a_{2}u_{2}^{*}v_{2}^{*}N_{2}}}.
\end{align*}
This is effectively equivalent to setting each of the leaves (of the sheaf) corresponding
to a derivative equal to zero, as done in Figure 28. There are many
other global sections within our sheaf that are not the equilibrium
points; however, the only global sections that correspond setting
the leaves equal to zero are the equilibrium points. 

\begin{figure}

\begin{centering}
\includegraphics[scale=0.4]{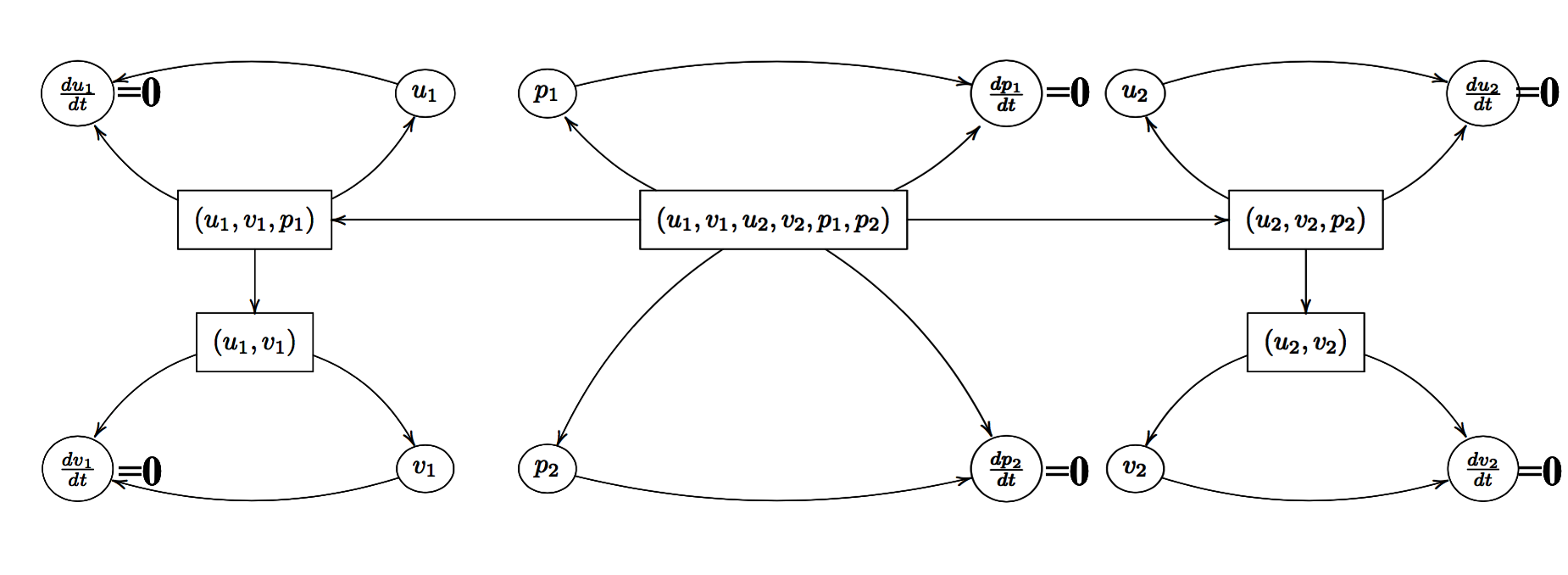}\caption{A Global Section of Extended Goodwin Model}

\par\end{centering}

\end{figure}

Recall that there is exactly \emph{one} equilibrium for a single country
Goodwin model and that the equation for a country's Goodwin model
are slightly different from the original Goodwin model; however, the
equilibrium points and equilibrium analysis remain unchanged. Since
the model for two countries engaged in horizontal trade includes two
single country models as slightly modified Goodwin models that still
poses the same equilibrium point, this means that each country's equilibrium
$(u_{i},v_{i})$ is determined from the outset. Further, the pricing
structure introduced by Ishiyama has these two countries $(u_{1},v_{1})$,
$(u_{2},v_{2})$ always uniquely determine commodity prices. Thus
there is exactly one equilibrium. Always. Recall also that the Goodwin
model for a single country is a conservative system, where level curves
form around the equilibrium point $(u_{i}^{*},v_{i}^{*})$. While
the countries engaged in trade at the equilibria do not change, Ishiyama
shows via computer simulation the dynamics are not quite as straightforward.
In fact, they are regularly chaotic. 

The questions of equilibria for the extended Goodwin model, while
vitally important from a dynamical systems perspective, are not the
interesting questions that can be asked from a sheaf theoretic perspective.
Since there is only one equilibrium point, it is unlikely that the
system will ever be in equilibrium, nor will it converge to an equilibrium
point, not much more can be said for this global section. More interesting
questions, both mathematically and economically - from a theoretical
and policy standpoint - are that of local sections, and their extensions.
For example, what happens if we decide to fix $(u_{1},v_{1})$? That
is, assert that these values in country one are known. Given this,
can we say anything that is happening in country two? Are there any
constraints that country one forces upon country two via the equations
in our system? If we have a local section at $(u_{1},v_{1})$, how
far can we extend this section? To the price-trading scheme space?
To $(u_{2},v_{2})$ or to the entire system? Does this allow us to
inferences about whether the Goodwin model has broken down in country
two? 

We now turn to a local section analysis on our sheaf. The main questions
we will be asking are: 
\begin{itemize}
\item How far can a local section be extended? 
\item What constraints does it have on the other variables? 
\item For some given choices, will all variables remain free, or will some
variables be required to be determined -\emph{ i.e.} not free - based
on how these local sections are chosen?
\end{itemize}
We will consider fixing $(u_{1},v_{1})$ and observing its effects
on $(u_{2},v_{2})$. The process of extending this local section amounts
to something of a ``diagram-chase,'' whereby we ``chase'' the
local section across the sheaf and thus see how far it can be extended. 

We will first start by noting the number of degrees of freedom in
each sub-diagram. In the sub-diagrams describing country one and country
two respectively, there are three degrees of freedom present. Each
sub-system has two equations and five variables, which can be seen
from the equations describing each countries individual dynamics:
\begin{align*}
\dot{v_{i}}= & v_{i}(t)\left(\frac{1}{\sigma_{i}}-(\alpha_{i}+\beta_{i})-\frac{u_{i}(t)}{\sigma_{i}}\right)\\
\dot{u}_{i}= & \frac{u_{i}(t)}{p_{i}(t)}\left(-(\alpha_{i}+\gamma_{i})+(\rho_{i}v_{i}(t))\right).
\end{align*}
Hence, we can freely pick three out of the five variables of the set
$\{u_{i},\dot{u}_{i},v_{i},\dot{v}_{i},p_{i}\}$; after this selection,
the final two will be fully dependent on the three picked. In our
sub-diagram describing the trading scheme, six degrees of freedom
are present. Again, there are two equations and eight variables present;
after picking six variables freely, the final two variables not picked
are determined. 

\begin{figure}

\begin{centering}
\includegraphics[scale=0.4]{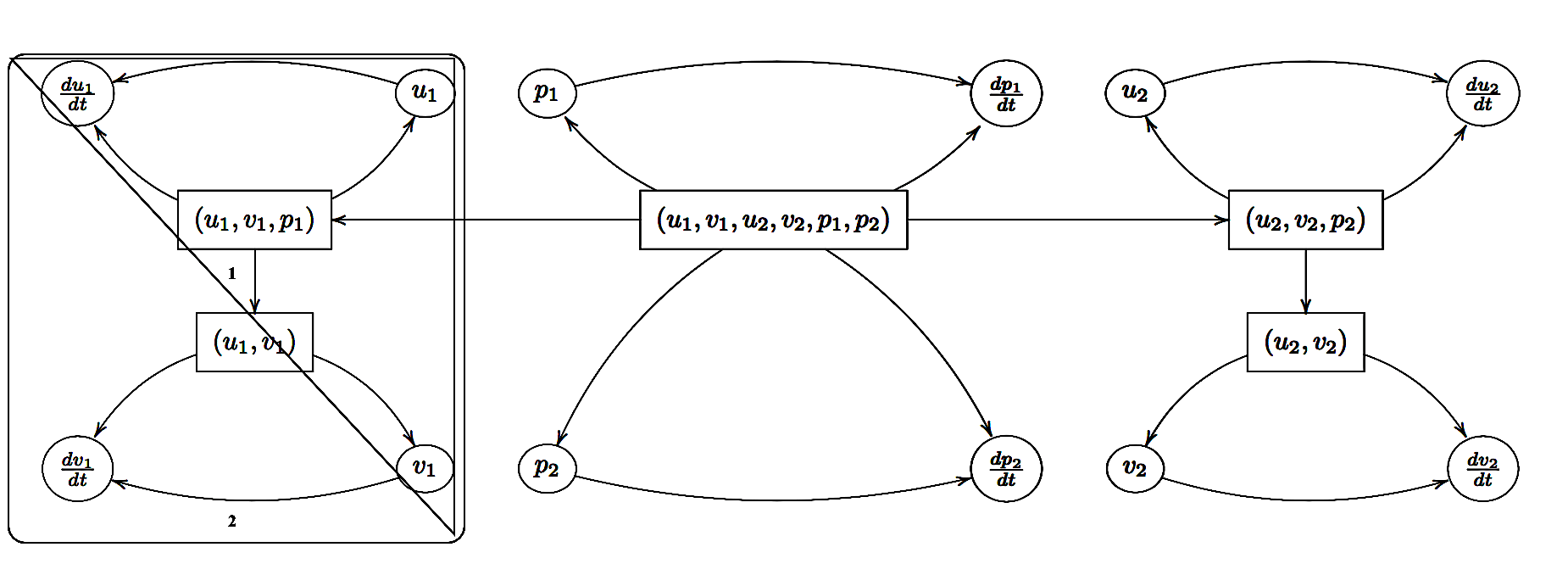}\caption{Extending Local Section of $(u_{1},v_{1})$}

\par\end{centering}

\end{figure}

With this in mind, we will begin by placing a local section on country
one and assert that three of the five variables present within the
system are known. Without loss of generality, we assert that the variables
$u_{1},v_{1},$ and $\dot{u}_{1}$ are known; $\dot{v}_{1}$ and $p_{1}$
are now fixed and determined based on the functional relationship
described above. As Figure 29 shows, we have gone from extending the
local section from a few variables to the entire system within country
one. 

All of country one's variables are completely known, since we either assert
they are known, or they can be determined based on the functional
relationships presented with country one's system. We will now try
to extend to the price-trading system, found in the middle of the
sheaf. Notice that three out of the five variables in country one
are also present in the price-trading system: $u_{1},v_{1},$ and
$p_{1}$. Since these three variables are now known in country one,
they carry over to the price-trading system. The local section has
partially been extended to the price-trading scheme. Recall that there
are six degrees of freedom within this portion of the diagram. Since
we have eliminated three of those degrees of freedom with the three
known variables from country one, we are left with three degrees of
freedom. Thus from our set of free variables $\{\dot{p}_{1},p_{2},\dot{p}_{2},u_{2},v_{2}\}$,
choosing three freely will then determine the remaining two. Without
loss of generality, we assert that $\dot{p}_{1},p_{2},$ and $\dot{p}_{2}$;
$u_{2}$ and $v_{2}$ are now completely determined by their functional
relationships with the other variables. The local section has thus
been fully extended to the price trading scheme \emph{and} partially
extended to country two. All eight variables are ``locked down,''
as shown in Figure 30. 

\begin{figure}
\begin{centering}
\includegraphics[scale=0.4]{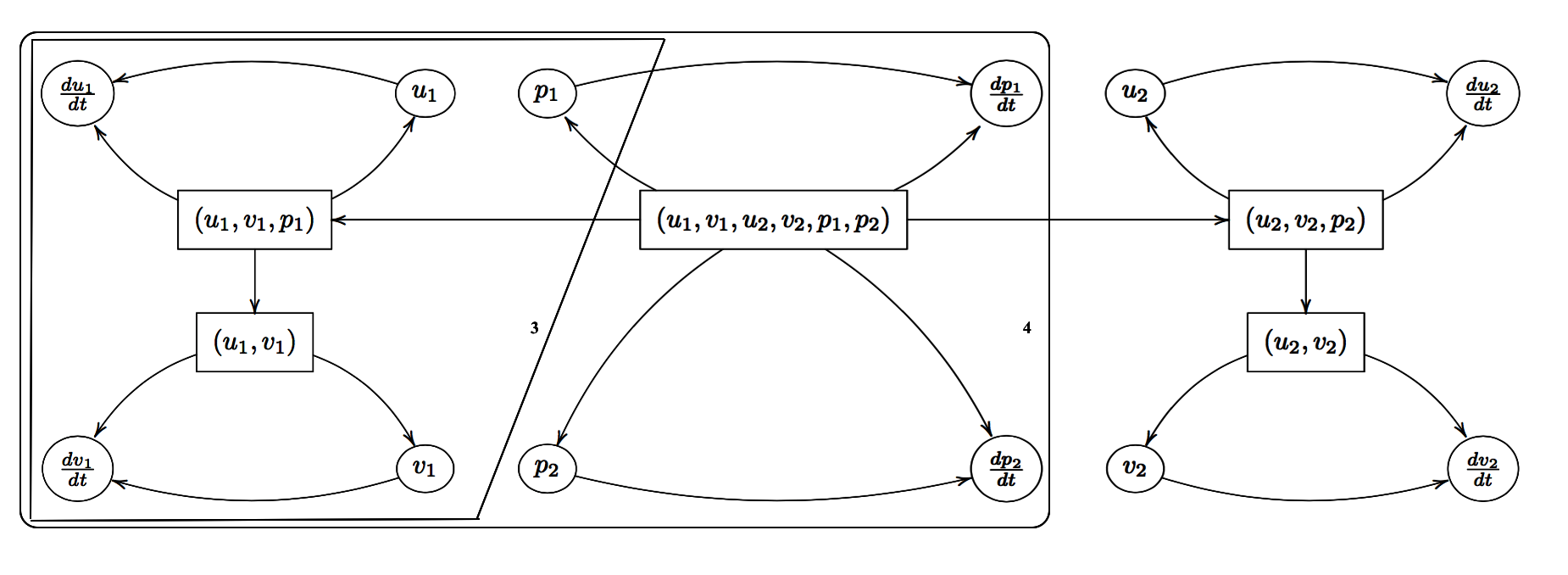}\caption{Extending Local Section of $(u_{1},v_{1})$ to Price-Trading System}

\par\end{centering}

\end{figure}

We turn now to country two to see if the local section can be extended
further. Recall that from the outset, there are three degrees of freedom
present within the country two, for the exact same reason as country
one. Similarly, there is an overlap of three of the variables found
within country two as there are in the price-trading system. As we
move across to country two, these variables which are now determined
also move across. This is the process of extending the local section
to part of country two. The three degrees of freedom that were present
within country two are now used up, leaving $\dot{u}_{2}$ and $\dot{v}_{2}$
completely determined. As Figure 31 shows, from this theoretical framework,
the local section has been extended across the entire system. 

\begin{figure}
\begin{centering}
\includegraphics[scale=0.4]{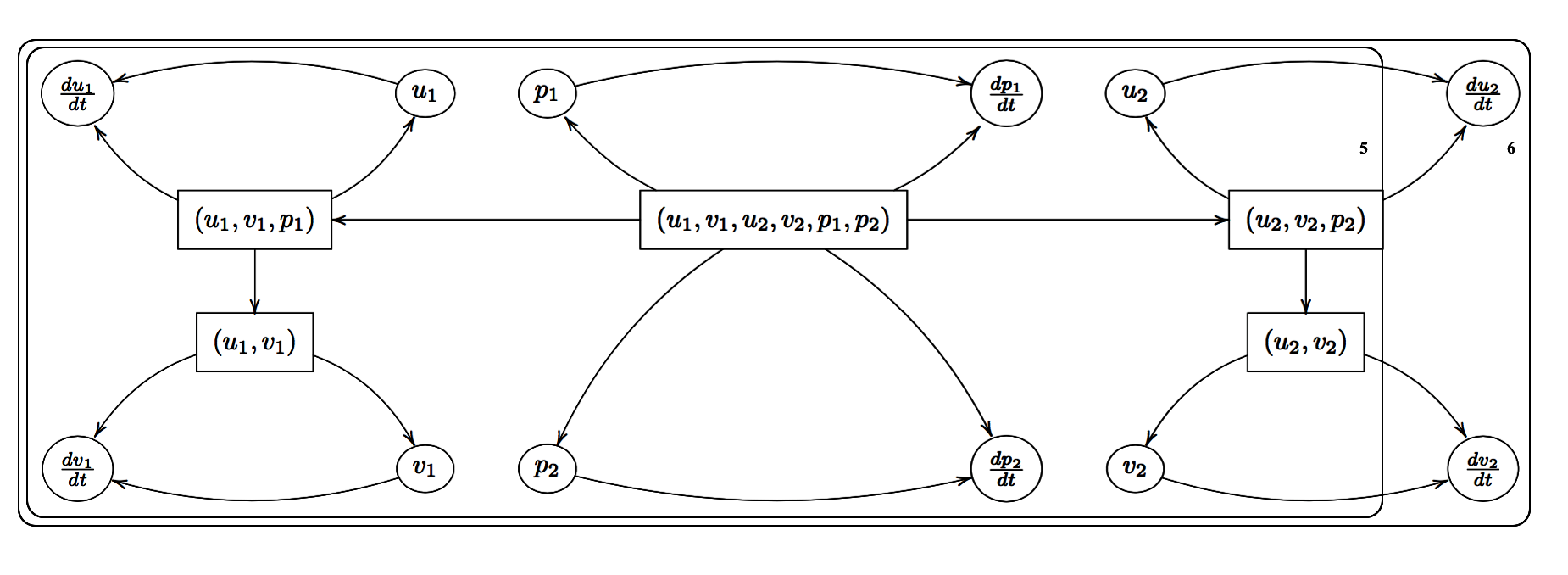}\caption{Extending Local Section of $(u_{1},v_{1})$ to Country two}

\par\end{centering}

\end{figure}

This analysis by employing sheaves is powerful, to say the least.
First, it \emph{is} possible for all this analysis to be done without
the aid of a dependency diagram simply using the equations. However,
the sheaf diagram helps represent the system graphically, making the
diagram chase an easier task. Similarly, because each arrow is a \emph{function
}- a product of the sheaf - and not just a relation, this local section
diagram chasing using the sheaf is equivalent to conducting the analysis
on the equations themselves. Further, this analysis is powerful since,
by successfully extending the local section across the entire system,
we can say the following: country one, having a good understanding
of their Goodwin cycle, and where they are within it can infer quite
a bit about what is happening in the rest of the system. Further,
with a little bit of information about the price-trading structure,
country one can fully determine country two's Goodwin cycle. Through
``locking down'' $u_{2}$ and $v_{2}$, country one knows where
country two is within it's Goodwin cycle \emph{and }because\emph{
}$\dot{u}_{2}$ and $\dot{v}_{2}$ are completely determined, they
know how country two is moving through their cycle. Thus, this local
section can be extended to a global section in our sheaf and system.

\subsection{Applications and Policy Implications}

The discussion and analysis conducted above were quite theoretical.
However, sheaves and their applications aren't restricted to this
theoretical level. This analysis extends quite nicely to a more practical
discussion. We saw at the end of the last subsection the power of
successfully extending local sections across a sheaf. Since we can
theoretically \emph{always }extend the local section, we should be
able to in practice. In theory, theory and practice are the same.
In practice, this is not always the case. Consider two countries each
with accurate data about the different variables in the extended Goodwin
model. Based on our analysis above, we should be able to fully extend
a local section $(u_{1},v_{1})$ to the entire system, making it a
global section, thereby accurately determining values of variables
in country two. If we attempt to extend this local section and practice,
there are two outcomes: there is disagreement based on the data or
there isn't. If there is disagreement on the data, this is valuable
insight; it tells us that the model we have constructed is not as
good as we thought! It forces us to reconsider the variable relationship
structures we've built out using our differential equations. If the
data agree, great! Our model appears to be a good representation of
the phenomena we are trying to model. Still more can be said.

For the remainder of this subsection, suppose that our model \emph{is
}good. The following represents the power of extending local sections
in practice. Consider two countries are engaged in horizontal trade.
Each has their own separate Goodwin model ``running'' within their
economy. If country one has a strong understanding of its own Goodwin
cycle, and understands a little bit of the trading structure, country
one can completely determine the Goodwin cycle running in country
two. This is what extending a local section of $(u_{1},v_{1})$ to
the entire system effectively amounts to. Country one, could for example,
control its own Goodwin cycle such that through trade, its dynamic
interaction with country two could lead to positive or negative results.
For example, it could manipulate its own Goodwin cycle in such a way
that country two experiences negative effects, hindering growth; alternatively,
it could help country two as well. 

This application of the power of local sections is unlikely at best
from an economic perspective. If a country is looking to help or hurt
a fellow country, other programs or policies exist which have that
specific goal. More likely than not, the power of local sections comes
from a data collection and synthesis standpoint. For example, imagine
two countries who are again engaged in horizontal trade. For the sake
of this example, let us say that country one is a developed nation
and its agencies that collect and report macroeconomic related statistics
are trustworthy and reliable, while country two, a developing nation,
is not trustworthy with its statistics reporting. Perhaps, country
one has cause to be doubtful of their reports, or country two simply
does not report this data. Given a strong understanding of its internal
Goodwin Model and some reliable data about prices of goods exchanged,
country one can effectively ``peer inside'' country two. Country
one can accurately gain an understanding of where country two is in
their Goodwin cycle as well as where they're going \emph{and }where
they've been. Understanding the Goodwin cycle of country one's trading
parter can help them make more informed dynamic policy decisions.

\section{Conclusion and Further Research}

In this paper we have explored the extended Goodwin model from a sheaf
theoretic approach. This paper has also introduced the concepts of
conservative dynamical systems, the original Goodwin model, and the
foundations of sheaf theory applied to systems of equations and differential
equations. Finally, we considered the extended Goodwin model, developed
by Ishiyama, which introduces horizontal international trade between
two countries. Sheaf theory provides a new avenue to conduct analysis
on differential equations by going beyond typical equilibrium and
dynamic systems analysis. Through considering local sections of sheaves,
and their extensions, theoretical inferences can be made from different
``vantage points'' within the model. Further, it allows for a suitable
``check'' on the theoretical model. Since local sections of $(u_{1},v_{1})$,
in the case of the extended Goodwin model can always be extended,
theoretically, they ought to be able to extend in practice; if not,
the model is not as good as we originally thought.

Given the novelty of this type of research approach, there are multiple
avenues for further research to build upon what has been started here.
The Goodwin model beautifully and simplistically models cycles in
an economy, making it a good candidate for updating and expanding. 

One avenue for future research concerns expanding Ishiyama's theoretical
framework to more than two countries. From looking at the sheaf
diagram, an extension to three countries engaged in horizontal pairwise
trade simply involves a ``cut and paste'' approach to build out
the sheaf. Extending to $n$ countries generalizes in a similar way,
generating a network of differential equations, described diagrammatically
by a sheaf. Given this new network, do the dynamics still appear to
be chaotic? Further, does the shape of the network result in different
dynamics? For example, if country one trades with country two and
three, the equilibrium prices and economic state in country one appears
to be independent of whether countries two and three trade with each
other; however the dynamics may change. The power of local sections
can be similarly applied to this extended network, and similar questions
arise. Given that two of country two's variables were completely determined
through country one's choice of $(u_{1},v_{1})$, how does this generalize
to a larger network?

A second approach to further research involves moving away from the
theoretical framework and into the application of data. With regard
to domestic countries, there is an ample amount of data available
with which the hypotheses presented in the two country framework could
be tested. 

Finally, the sheaf theoretic approach need not be confided to updated
Goodwin models. This approach is sufficiently general and malleable
that it can be applied to other theories of economic growth and trade,
supply chain evolution, or theories of financial markets. These problems
remain to be attacked in future areas of research.  \pagebreak{}

\bibliographystyle{plain}
\addcontentsline{toc}{section}{\refname}\bibliography{pcthesis}

\end{document}